\documentclass[10pt]{article}
\usepackage[utf8]{inputenc}
\usepackage{amsfonts,amsmath,amssymb,amsthm}
\usepackage{graphicx,verbatim}
\usepackage{bbm}
\usepackage{float}
\usepackage{subfig}
\usepackage{xcolor}
\usepackage[hidelinks]{hyperref}

\RequirePackage{geometry}
\newcommand\blfootnote[1]{%
  \begingroup
  \renewcommand\thefootnote{}\footnote{#1}%
  \addtocounter{footnote}{-1}%
  \endgroup
}

\newcommand{\pgap}{\vspace{0.3cm}}

\newcommand{\bR}{\mathbb{R}}
\newcommand{\bE}{\mathbb{E}}
\newcommand{\ep}{\epsilon}
\newcommand{\bI}{\mathbbm{1}}
\newcommand{\bP}{\mathbb{P}}
\newcommand{\bZ}{\mathbb{Z}}

\newcommand{\CN}{\mathcal{N}}
\newcommand{\CL}{\mathcal{L}}
\newcommand{\CH}{\mathcal{H}}
\newcommand{\CF}{\mathcal{F}}
\newcommand{\CR}{\mathcal{R}}
\newcommand{\CB}{\mathcal{B}}

\newcommand{\vol}{\mathrm{vol}}

\newcommand{\hess}{\mathrm{Hess }}
\newcommand{\tf}{\Tilde{f}}
\newcommand{\barf}{\Bar{f}}
\newcommand{\cov}{\mathrm{Cov}}
\newcommand{\tr}{\mathrm{Tr}}
\newcommand{\Var}{\mathrm{Var}}

\newcommand{\bOne}{\mathbf{1}}
\newcommand{\bfe}{\mathbf{e}}
\newcommand{\td}{\text{d}}
\newcommand{\bara}{\Bar{a}}
\newcommand{\barv}{\Bar{v}}

\newtheorem{proposition}{Proposition}[section]
\newtheorem{theorem}[proposition]{Theorem}

\newtheorem{lemma}[proposition]{Lemma}

\newtheorem{assumptions}[proposition]{Assumptions}

\newtheorem{definition}[proposition]{Definition}

\title{ High local maxima of stationary smooth Gaussian fields}
\author{Dmitry Beliaev \footnote{Mathematical Institute, University of Oxford, UK } \and Akshay Hegde \footnote{Department of Mathematics, National University of Singapore, Singapore} }
\date{}

\begin{document}

\maketitle

\begin{abstract}
Consider the point process (in $\bR^d$) of local maxima of smooth Gaussian fields, with sufficient decay of correlation at infinity, above a level $u$. We show that this point process, rescaled appropriately, converges weakly to a Poisson point process in the limit $u \to \infty$. Our proof relies on the classical observation that simple point processes are characterised by avoidance probabilities (i.e. $\bP(\eta(B)=0)$ for a point process $\eta$ and Borel set $B$).
Then we approximate avoidance probability with the excursion probability, where the latter is well studied. 
Second main result is a quantified version of the Poisson convergence of high local maxima of the Bargmann-Fock field in $\bR^2$. We prove that, for Bargmann-Fock field in two dimensions, the total variation distance between a Poisson random variable and the number of local maxima of the field above a threshold $u$ in an $R \times R$ box in $\bR^2$ decays like $\exp(- \beta u^2)$, for some fixed $\beta >0$. As an immediate consequence, when the level $u$ is a function of $R$ such that $u(R) \to \infty$ and $u(R)/ \sqrt{\log R} \to 0$ as $R \to \infty$, we have a quantitative central limit theorem for the number of high local maxima. 
The proof is based on the Chen-Stein method for quantitative Poisson approximation. We produce a close coupling of a stationary smooth field and its Palm version, which might be of independent interest.
\end{abstract}


\blfootnote{Emails: \href{mailto:belyaev@maths.ox.ac.uk}{belyaev@maths.ox.ac.uk}, \href{mailto: akshay.hegde@nus.edu.sg}{akshay.hegde@nus.edu.sg} }

\tableofcontents

\section{Overview}

Poisson process convergence of high excursion sets/points of random functions is a classical topic \cite{leadbetter_extremes_1983}.  In one dimension, the number of upcrossings above the level $u$ of a smooth Gaussian process, when appropriately scaled, converges weakly to the homogeneous Poisson point process as $ u \to \infty$ \cite[Chapter 9]{leadbetter_extremes_1983}. 
In dimension 2 or more, ``exit points" above the level $u$ of a smooth Gaussian field, scaled appropriately, converges weakly to the homogeneous Poisson point process as $u \to \infty$ \cite[Section 15]{piterbarg_asymptotic_1996}.

\pgap 

In all of the examples above, the field's correlation decay required is faster than $1/\log$ of the distance between the points. 
A somewhat related set of results includes limit theorems for extremal processes for the class of processes with the Markov property, like the Gaussian free field, and  branching Brownian motion.
Arguin et al. \cite{arguin_extremal_2013} showed that the extremal process of branching Brownian motion converges weakly to a clustered Poisson process. 
Oleskar-Taylor, Sousi \cite{sousi_chenstein_2020} showed that the high points (level above $\alpha \bE [\rm{maxima}], 0<\alpha_0<\alpha$) of the discrete Gaussian free field in $d \geq 3$ converge in total variation distance to an independent Bernoulli process on the lattice.
In essence, we can expect some Poisson limit for an extremal process if either the covariance decays fast enough at infinity or there is some Markov property.

\pgap 

Let $f: \bR^d \to \bR$ be a smooth, centered stationary Gaussian field with covariance function $r(x,y)=\bE[f(x)f(y)]$. We are interested in the critical point structure of the random function $f$. Some of the example models of the stationary smooth fields are the following (see \cite{hegde2026geometry} for motivation and background). A \textit{Bargmann-Fock field} is a smooth Gaussian field $f$ with the covariance function $r(x,y)= \exp(-\|x-y\|^2/2)$. Another example is \textit{monochromatic random field}, whose covariance function is a Fourier transform of a symmetric measure supported on the unit sphere $\mathbb{S}^{d-1} \subset \bR^d$. A special case is the \textit{random plane wave} (also known as the Berry's random plane wave) where $d=2$, and the spectral measure is the uniform measure on $\mathbb{S}^1$. In that case, the covariance function becomes $r(x,y)= J_0(|x-y|)$, where $J_0$ is the zeroth Bessel function of the first kind.   

\pgap 

Our contribution is the study of local maxima of smooth Gaussian fields above level $u$. The first main result (Theorem \ref{thm-weak-conv}) states that rescaled high local maxima converges in law to Poisson point process as the level $ u \to \infty$. Our result extends \cite[Theorem 3.3.2]{qi_excursion_2022}, which includes important models such as random plane waves (RPW) and other monochromatic random waves in dimensions $d \geq 2$.  

\pgap 

This weak convergence of high critical points raises natural questions: Considering local maxima above level $u$ in a window $[-R,R]^d$, what is the Wasserstein distance between this point process and a suitable Poisson point process, given a (pseudo-)metric ? What is the role of the rate of decay of the covariance kernel $r(x,y)$ of the field? In Theorem \ref{thm-number-count}, we show that the total variation distance (Wasserstein distance with pseudo-metric identically zero) between the number of high critical points of the Bargmann-Fock field and a Poisson variable is exponentially small in the level $u$. The main idea of the proof is as follows.

\begin{figure}[h]
    \centering
    \includegraphics[width=\linewidth]{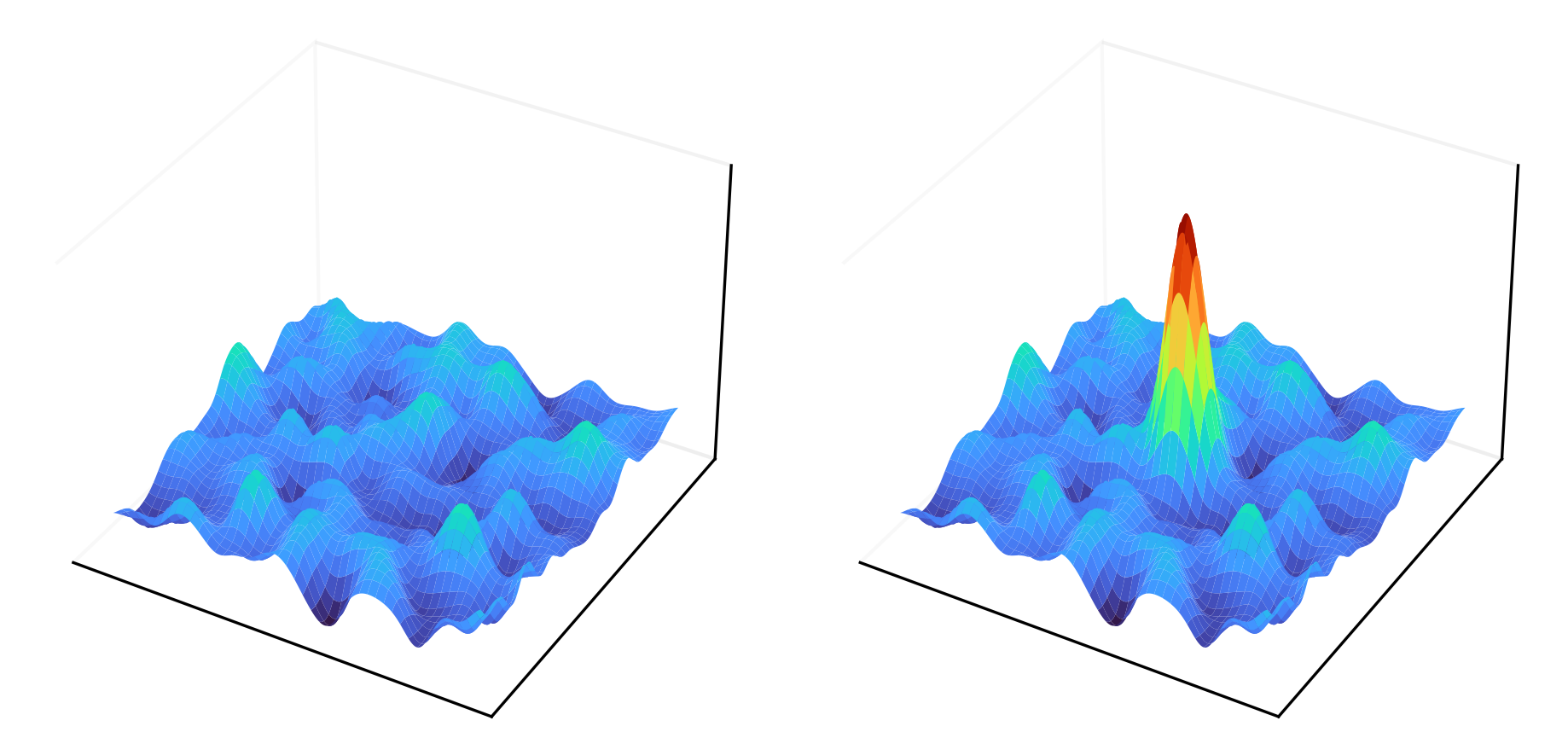}
    \caption{(Left) Bargmann-Fock field in a 20 by 20 box. (Right) Palm coupling of the field with maxima at origin with height at least 20. }
    \label{fig:Palm-coupling-BF}
\end{figure}

\begin{figure}[htbp]
    \centering
    \begin{minipage}{0.48\textwidth}
        \centering
        \includegraphics[width=\textwidth]{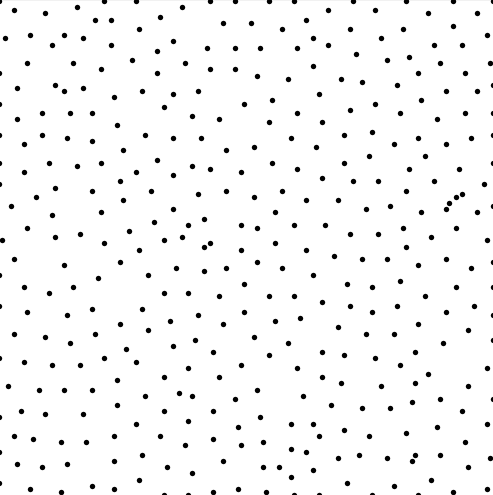}
        \caption*{(a) RPW: All maxima in a 20 by 20 box.}
    \end{minipage}
    \hfill
    \begin{minipage}{0.48\textwidth}
        \centering
        \includegraphics[width=\textwidth]{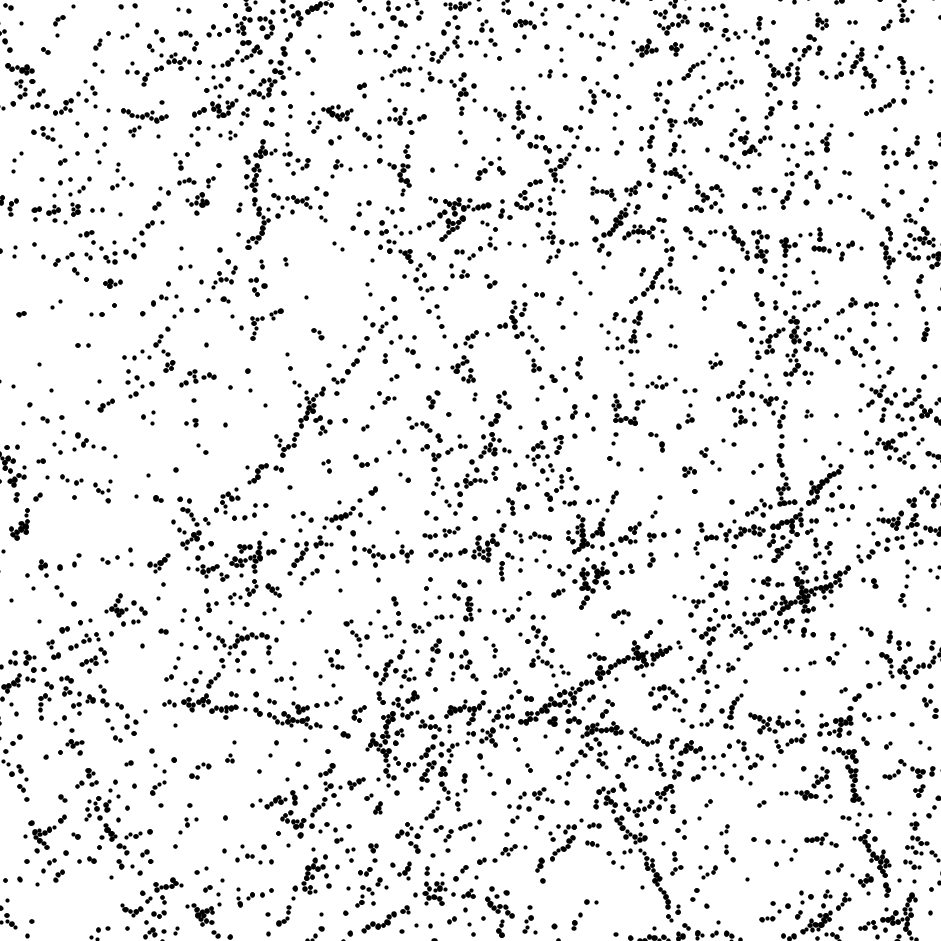}
        \caption*{(b) RPW: Maxima above 2.5 in a 200 by 200 box.}
    \end{minipage}

    \vspace{0.8em}
    \begin{minipage}{0.5\textwidth}
        \centering
        \includegraphics[width=\textwidth]{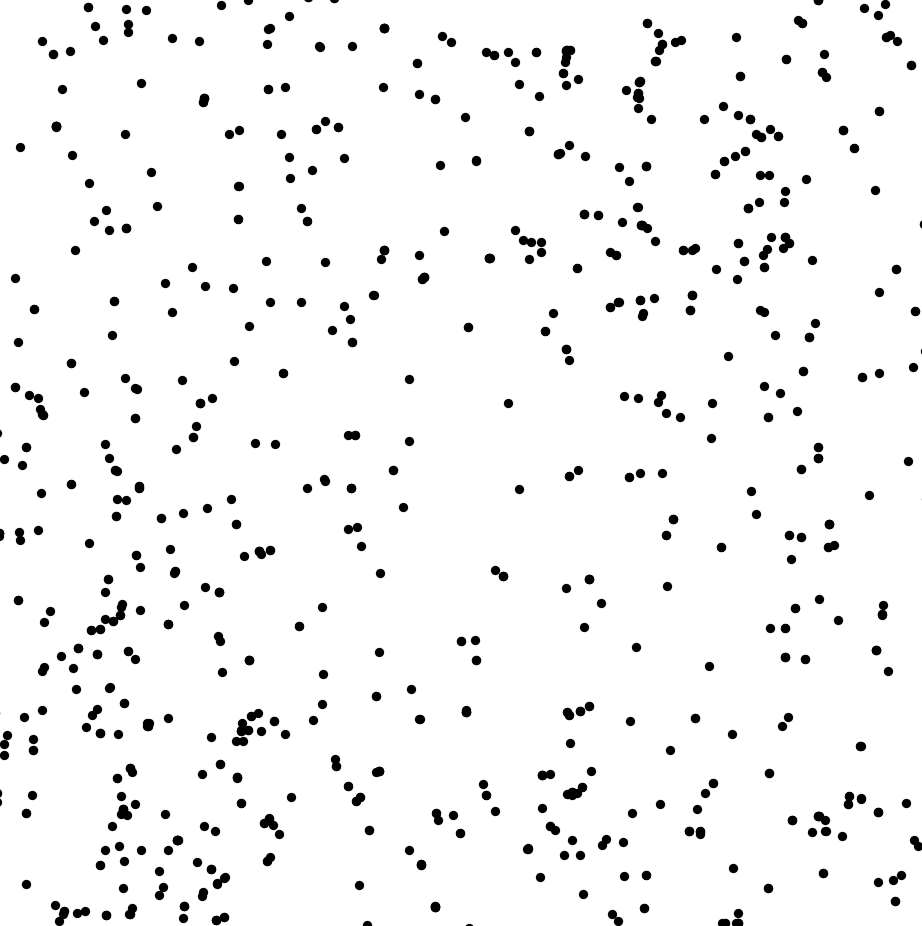}
        \caption*{(c) RPW: Field above 4.5 in 2000 by 2000 box, represented by dots. }
    \end{minipage}

    \caption{Hierarchy of maxima in the Random Plane Wave model. All three pictures are from the same sample, but at different scale and level.}
    \label{fig:rpw_triangle}
\end{figure}

\begin{figure}[ht]
    \centering
    \subfloat{{\includegraphics[width=6cm]{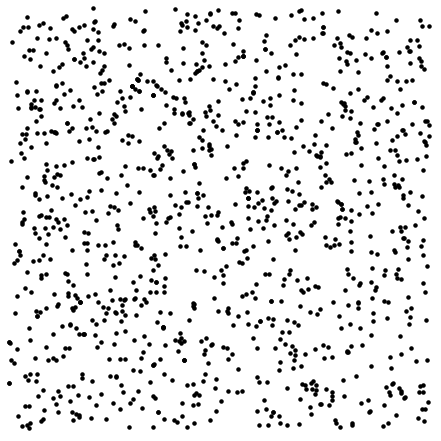} }}%
    \qquad
    \subfloat{{\includegraphics[width=6cm]{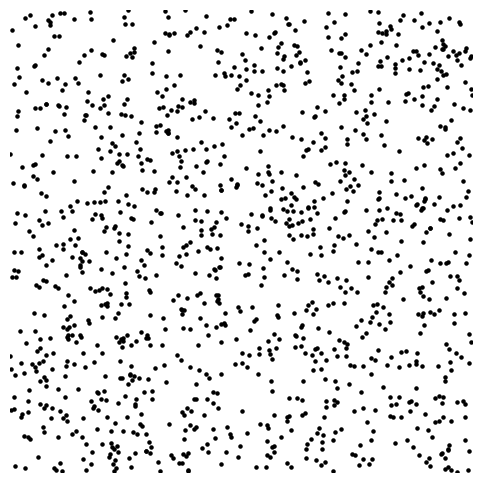} }}%
    \caption{(Left) Bargmann-Fock critical points in dimension 2. (Right) Poisson point process with the same intensity.}
    \label{fig:BF-crit-Poisson}%
\end{figure}

\begin{figure}%
    \centering
    \subfloat{{\includegraphics[width=6cm]{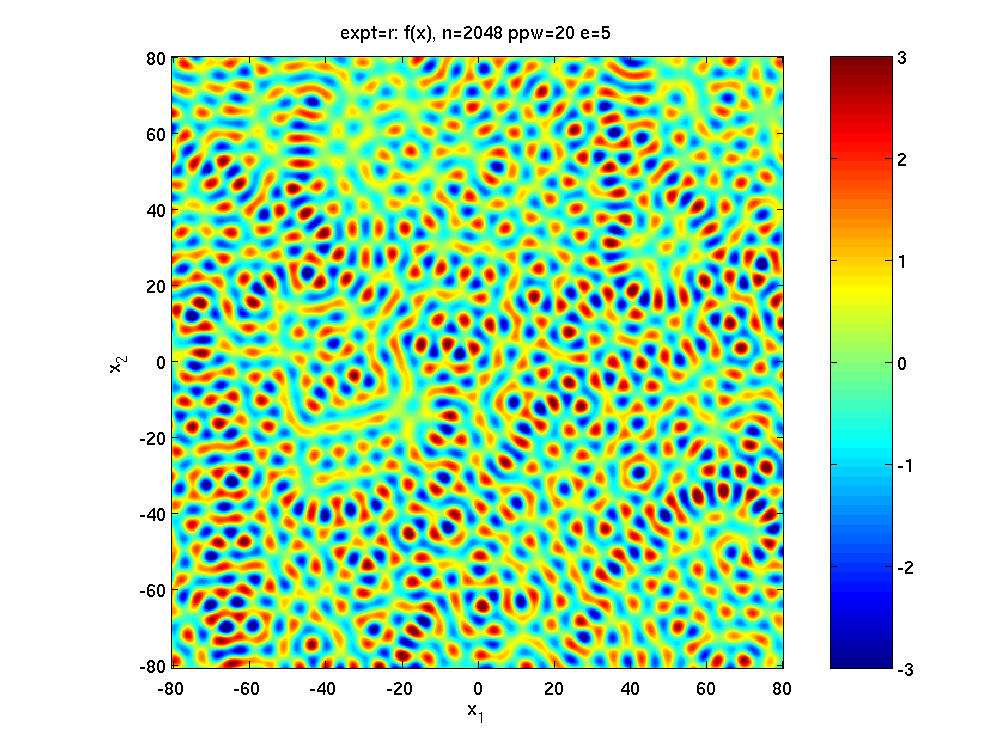} }}%
    \subfloat{{\includegraphics[width=6cm]{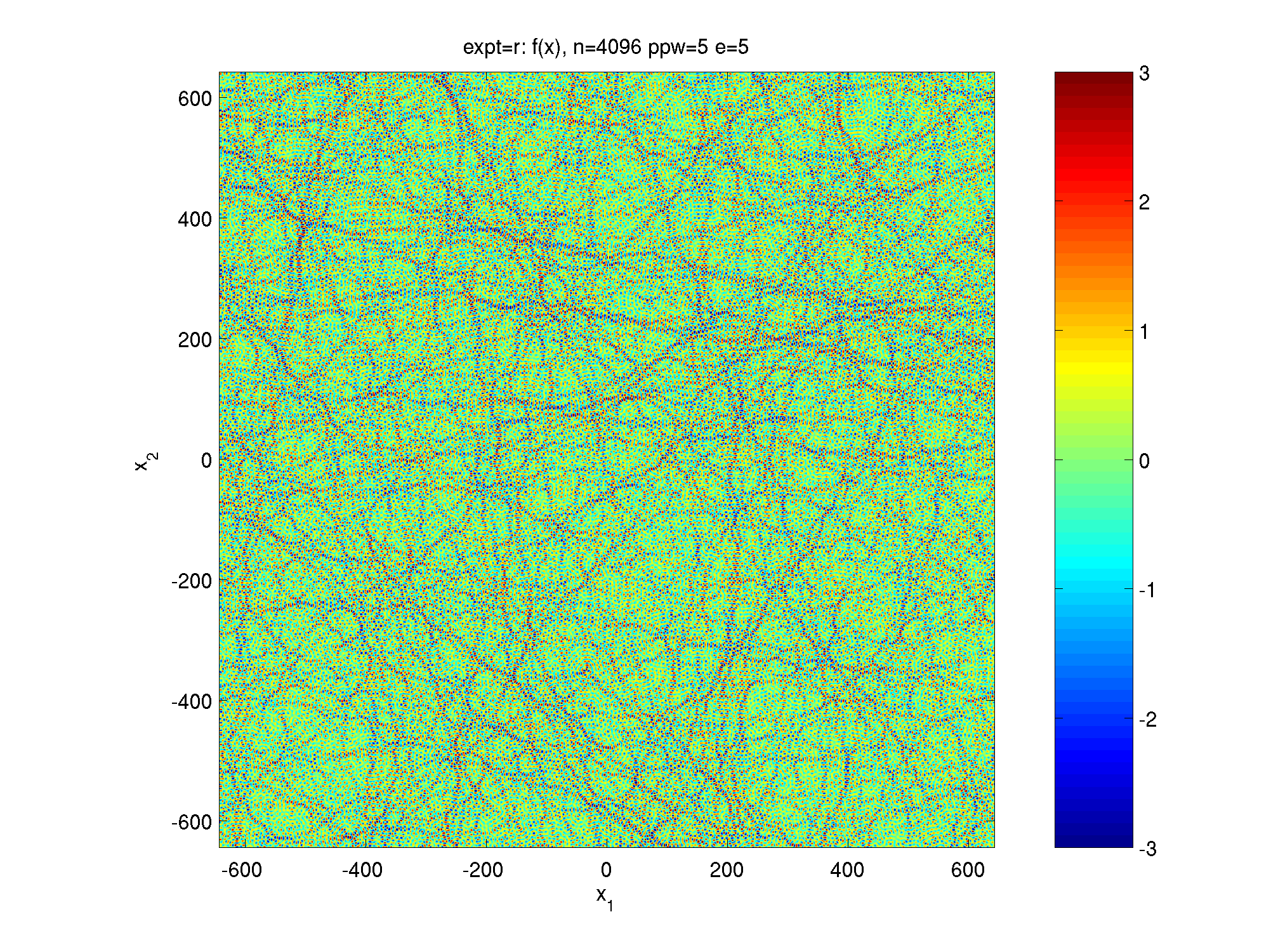} }}%
    \caption{(Left) A sample of RPW in a box of sidelength around 200 wavelengths. (Right) The same with around 2000 wavelengths. Picture by Alex Barnett (Link: \href{https://users.flatironinstitute.org/~ahb/rpws/}{users.flatironinstitute.org/ \textasciitilde ahb/rpws/}) }%
    \label{fig:RPW-filaments}%
\end{figure}

\pgap

In the random plane wave (RPW) model, the spatial organization of the maxima exhibits a striking hierarchy. When all maxima are considered, their configuration resembles a quasi-lattice pattern, reflecting the strong short-range repulsion between critical points an effect rigorously analyzed by Beliaev, Cammarota, and Wigman \cite{beliaev_two_2019}. However, restricting attention to higher levels, say those exceeding two to three standard deviations, a filamentary network emerges, where maxima tend to cluster along elongated ridges of high field intensity. At even more extreme levels, around four and a half standard deviations and beyond, this rigidity breaks down: the high peaks become sparse and spatially uncorrelated, following statistics close to a Poisson point process. This transition —from local order to apparent randomness- is clearly visible in Figure \ref{fig:rpw_triangle}.

\pgap 

In sharp contrast to the Random Plane Wave (RPW) model, the critical point structure of the Bargmann–Fock field shows little sign of spatial rigidity. When all critical points are plotted, they exhibit an almost completely disordered, Poisson-like pattern, with no visible short-range repulsion or lattice organization (see Figure \ref{fig:BF-crit-Poisson}). This reflects the much faster decay of correlations in the Bargmann–Fock field compared to the oscillatory, long-range structure of RPW. As a result, the Bargmann–Fock field behaves more like a locally independent Gaussian surface—its critical points scattered almost as if placed at random.

\pgap 

The distribution of critical points of random functions is an interesting topic, aside from its applications in science and engineering. Systematic study of critical points of smooth Gaussian fields dates back to the late 1960s \cite{nosko_local_1969}. In the 1970s, the topic was explored further by Belyaev and Piterbarg, among others. See Adler's 1980s book \cite{adler_geometry_2010} for developments till the 80s. For an overview of results, see Piterbarg's monograph \cite{piterbarg_asymptotic_1996}, and Adler and Taylor's book \cite{adler_random_2009} (for a more recent one). The following paragraph lists some interesting results in the last 15-20 years.  

\pgap 

 Characterising the distribution of the number of critical points in a bounded domain of a general smooth Gaussian field is a non-trivial task currently out of reach. Theoretically, the mean number and higher moments can be estimated from the Kac-Rice formulas. At a large scale, variance estimates for specific models like random spherical harmonics are available now \cite{cammarota_fluctuations_2017}. 
 Recently, Ancona, Gass, Letendre, and Stecconi \cite{ancona_zeroes_2025} (building upon previous works of \cite{gass_number_2023,cuzick_conditions_1975}) computed all cumulants of the number of critical points and, as an application, proved the law of large numbers and the central limit theorem (CLT). Previously, the Wiener chaos decomposition was employed to show the CLT for certain Gaussian fields \cite{nicolaescu_clt_2017}.
The spatial structure of the critical points, the high critical point/ excursion set, was studied in detail in the latter half of the 20th century. The study of the local structure of RPW critical points was initiated in \cite{beliaev_two_2019}, which inspired the generalisation of the result to other models \cite{beliaev_repulsion_2020,ladgham_repulsion_2023}.

\pgap 

\paragraph{Acknowledgment:} The research was carried out during AH's doctoral study in University of Oxford. We are grateful to Igor Wigman and Ben Hambly for comments and suggestions, including an error in Theorem \ref{thm-weak-conv} in the earlier draft. AH thanks Yogeshwaran for the helpful discussions regarding quantitative Poisson approximations. 
\pgap

\section{Main results}

\subsection{The Poisson point process convergence}

Consider a $C^{2}$-smooth centered Gaussian field $f:\bR^d \to \bR$, with $d\geq 2$. Let $\bP$ be the associated probability measure and $\bE$ the expectation with respect to $\bP$. By  $r(x,y)=\bE[f(x)f(y)]$ we denote the covariance kernel of $f$. 

\pgap 

For $u>0$, we consider the point process of local maxima of $f$ above the level $u$ in $\bR^d$, i.e. all local maximum $z_0 \in \bR^d$ with $f(z_0) >u$.   
By a theorem of Adler from 1970's (see Theorem \ref{thm-adler-density}), the density of this process is 
\[
c u^{d-1} \exp(-u^2/2)(1+O(u^{-1})) \quad \text{ as } u \to \infty
\]
where the constant in $O(u^{-1})$ and $c$ depend only on the law of $f$. The density goes to zero as $u \to \infty$, hence we rescale the point process to have asymptotic density one. 
\pgap

Define $\Phi_u$, for $u>0$, to be the point process such that for each Borel set $B \subseteq \bR^d$, 
\[
\Phi_u(B)= \text{ number of local maxima of } f \text{ above level } u \text{ in } \mu(u)B
\]
where 
\[
\mu(u)= (2 \pi)^{\frac{d+1}{2d}} u^{\frac{1-d}{d}}\exp \left ( \dfrac{u^2}{2d} \right).
\]

 $\mu$ chosen in such a way that the density of $\Phi_u$ is asymptotically equal to $1$ (c.f. Theorem \ref{thm-adler-density}). We prefer to have an explicit scaling factor and asymptotic density rather than rescaling by the exact density, which does not have a simple closed form expression.  

\pgap

\begin{assumptions} \label{assumptions-2} We impose the following conditions on the Gaussian field $f$. 

\begin{enumerate}
    \item Centered $(\bE[f(x)]=0)$, stationary $(r(x,y)=r(x-y))$, normalised $(\bE[f(x)^2]=1)$ for all $x,y \in \bR^d$.
    \item Decay of correlation: $r(x)=o((\log \|x\|)^{-1})$ as $x \to \infty$. 
    \item The vector $(f(0), \nabla f(0))$ has density in $\bR^{d+1}$. In addition, either the vector $$(f(0), \nabla f(0), \nabla^2 f(0))$$ has density in $\bR^{(d+1)+d(d+1)/2}$ or $f$ is a monochromatic random wave (i.e. spectral measure of $f$ is supported on $\mathbb{S}^{d-1}$, the $(d-1)$-dimensional unit sphere.
    \item Local structure: $r(x)=1-\|x\|^2+ o(\|x\|^2)$ as $x \to 0$. 
\end{enumerate}

\end{assumptions}
The stationarity assumption is crucial to our argument. We utilise that the density of critical points is homogeneous in $\bR^d$ and also that asymptotic densities (like $\mu^{-1}$) of critical points are hard to come by for non-stationary fields. 
The decay of correlation assumption rate is optimal (i.e. necessary and sufficient) for Poisson process convergence for discretised versions (c.f \cite[Corollary 13.1]{piterbarg_asymptotic_1996}). Hence we believe that it is the same case for the smooth version as well. Note that the last condition is not very restrictive; it is just a convenient normalization since for any $C^2$-smooth  field $f$ satisfying other assumptions, there exists an invertible matrix $M$ such that $r(Mx)= 1-\|x\|^2+ o(\|x\|^2)$ as $x \to 0$. These normalisations (i.e. unit variance and local structure assumption) implies that the asymptotic density of $\Phi_u$ is one.

\pgap 

One observation regarding the covariance structure $r$ is that 
\[
r(x,y) <1 \quad \forall x \neq y. 
\]
This follows from stationarity of the field and the fact that $r(x) \to 0$ as $x \to \infty$. This is helpful when estimating the exceedance probability of the field over a large given threshold.

\begin{theorem}\label{thm-weak-conv}
    With the setup above and with the Assumptions \ref{assumptions-2} on the Gaussian field $f: \bR^d \to \bR$, we have 
    \[
    \Phi_u \to \Phi \quad \text{in law as } u \to \infty  
    \]
    where $\Phi$ is the Poisson point process with intensity measure as Lebesgue measure on $\bR^d$.
\end{theorem}

First, note that an invertible linear transform $T$ of a Poisson point process (with intensity measure $\lambda$)  is again a Poisson point process with new intensity measure $|\det (T)|\lambda$.
So rescaling the field to satisfy the last condition in Assumption \ref{assumptions-2} does not change the result. This allows us to work with suitable normalized fields. Next, the Bargmann-Fock field and monochromatic random waves for dimension $d\geq 2$ satisfy the assumptions. Indeed, the covariance kernels have decay rates $\exp(-\|x\|^2/2)$ and $O(\|x\|^{-1/2})$ for Bargmann-Fock and monochromatic random waves respectively. 

\pgap 

Another remark is that even though we have stated our result for the point process of local maxima of $f$, the same result holds even for \emph{all} critical points above level $u$. This is because $\mu(u)^{-d}$ is still the asymptotic density of critical points above level $u$. For example, for random plane wave model, by Example 3.15 of \cite{cheng_expected_2018} we have that intensity of local maxima above level $u$ is $c_1 \cdot u \exp(-u^2/2)$ as $ u \to \infty$ whereas the intensity of saddle points is asymptotically only $c_2 \cdot  u^{-1} \exp(-3u^2/2)$. 

\pgap 

For the Bargmann-Fock field in dimension 2, similar asymptotically explicit intensities for saddle points and local minima are calculated in Example 3.8 of \cite{cheng_expected_2018}. Hence, we believe that intensities of critical points of lower index (not the top index i.e. local maxima) above a high level $u$ is o-small of that of local maxima. This heuristic is further supported by the estimates in \cite[Theorem 11.7.2]{adler_random_2009} where the expected Euler characteristics of the excursion of $f$ above $u$ has the same asymptotics as that of local maxima intensity. By Morse theory, this is almost the same result we need.

\pgap 

Our motivation to consider local maxima of Gaussian fields above a level was due to the apparent filament structure of local maxima above 2 to 3 standard deviation of RPW field (see Figure \ref{fig:RPW-filaments}).

\pgap

Finally, we present a simple heuristic behind the statement of the theorem \ref{thm-weak-conv}. As the level $u$ tends to infinity, the density of the critical points above this level tends to zero and the typical distance between them tends to infinity. Since the covariance kernel goes to zero and is independent of $u$, for sufficiently large $u$ these points are so far away from each other that they are essentially independent of each other. This means that when we rescale them to have density $1$ they become close to a Poisson point process. Unfortunately, it is not that easy to turn this heuristic into a rigorous argument. So our proof follows a slightly different path.

\subsection{Quantitative convergence}

In this section, we quantify the convergence of high local maxima of a field converging to Poisson point process inside a box of size $R$. We will restrict ourselves to dimension $d=2$ in this section but we believe the result holds for any $d \geq 2$. Given a smooth Gaussian field $f$ on $\bR^2$, for $R,u>0$, define
\[
\Psi_{R,u}:= \Phi_u|_{D_{R,u}} \text{ where } D_{R,u}=[-R \mu(u)^{-1}/2,R \mu(u)^{-1}/2]^2
\]
i.e., we restrict the point process $\Phi_u$ of rescaled local maxima of $f$ above level $u$ to the domain $D_{R,u}$. Note that $\Psi_{R,u}$ still has density approximately one. 

\pgap 

For a point process $\eta$ on a domain $D \subset \bR^2$, denote by $|\eta|$ the number of points in $D$. Also, let $\CL(X)$ denote the law of a random variable $X$. We know that from Theorem \ref{thm-adler-density}, for large $u>0$, 
\[
\bE |\Psi_{R,u}| \simeq c R^2 u \exp(-u^2/2).
\]
We want to compare $\CL(|\Psi_{R,u}|)$ to a Poisson random variable with suitable parameter. Let $U_{R,u}$ denote a Poisson random variable with mean $\bE|\Psi_{R,u}|$.

\begin{theorem} \label{thm-number-count}
Let $R>0$ and $u>0$ such that $u \leq 2\sqrt{ \log R}$. Let $f: \bR^2 \to \bR$ be the Bargmann-Fock field and $\Psi_{R,u}$ be the point process associated to $f$ as defined above. Then, for $R>0$ large enough,   
    \[
    d_{TV}(\CL(|\Psi_{R,u}|),\CL(U_{R,u})) \leq c \exp(-\beta u^2)
    \]
    for some constants $c,\beta>0$ that does not depend on $u, R$.
\end{theorem}

The above theorem can be interpreted as a quantitative central limit theorem in certain cases. Observe that, if $U$ is a Poisson random variable with mean $\lambda$, then 
\[
\dfrac{U- \lambda}{\sqrt{\lambda}} \to Z \sim \CN(0,1) \text{ in law as } \lambda \to \infty. 
\]
Also, note that the total variation distance is shift and scale invariant. That is, if $P_{s,t}, Q_{s,t}$ denote that laws of random variables 
\[ \dfrac{X-s}{t}, \dfrac{Y-s}{t}\]
respectively, then 
\[d_{TV}(P_{0,1}, Q_{0,1}) = d_{TV}(P_{s,t}, Q_{s,t}) \quad  \text{ for all } s,t \text{ with } t\neq 0.\]

\pgap 

Hence we need to ensure that, in Theorem \ref{thm-number-count}, the level $u$ depends on $R$ such that $\bE[|\Psi_{R,u}|] \to \infty$ as $R \to \infty$.

\pgap 

\begin{theorem}[Quantitative central limit theorem] \label{thm-qclt}
Let $u=u(R)$ be a function of $R$ such that 
   \[ u(R) \to \infty \text{ and } \dfrac{u(R)}{\sqrt{\log R}} \to 0 \text{ as } R \to \infty.\]
   With assumptions and notations as in Theorem \ref{thm-number-count}, we have 
   \[
   d_{TV} \left ( \CL \left (  \dfrac{|\Psi_{R,u}|- \bE[|\Psi_{R,u}|]}{\sqrt{\bE |\Psi_{R,u}|}}\right), \CL \left ( \dfrac{U_{R,u} - \bE[|\Psi_{R,u}|]}{\sqrt{\bE |\Psi_{R,u}|}} \right) \right) \leq  c \exp(-\beta u(R)^2)
   \]
   for large enough $R$. Here the constants $c, \beta >0$ does not depend on $R$ or the function $u(R)$.
\end{theorem}

\textbf{Remarks:} 
\begin{itemize}
    \item Notice that $|\Psi_{R,u}|$ is asymptotically (as  $u \to \infty$) the number of local maxima of $f$ above the level $u$ in $[-R,R]^2$. This is because we scale the point process by $\mu(u)$ where $\mu(u)^{-2}$ is the \textit{asymptotic} intensity of the local maxima of $f$ above level $u$, and not scale by the exact intensity. We can rewrite Theorems \ref{thm-number-count} and \ref{thm-qclt} in terms of the number of local maxima of $f$ above level $u$ in $[-R,R]^2$, with same proof. Also, for the proof we need to spatially rescale the point process of local maxima.

    \item As pointed out in the remarks after Theorem \ref{thm-weak-conv}, the quantitative estimates in Theorem \ref{thm-number-count} and \ref{thm-qclt} also hold for the point process of all critical points of $f$ above level $u$ instead of just local maxima.

    \item  We conjecture that, for Gaussian fields with covariance kernel $r \in L^1(\bR^d)$, the upper bound of Theorem \ref{thm-number-count} holds with a suitable constant $\beta$ ,where the constant $\beta>0$ would depend on the law of the field $f$, such that the rate of decay of the covariance kernel and derivatives of the kernel at the origin. Also, notice in the proof that the exponential upper bound in Theorem \ref{thm-number-count} comes from the decay of the tail of the Gaussian random variable rather than the decay of the covariance of the Bargmann-Fock field.

\item Regarding the random plane wave (RPW) model, we believe we can get an upper bound for the total variation distance, which tends to zero as $R \to \infty$, although with a different rate than in the integrable kernel case. A potential technical obstacle to extending the result to the RPW model is the applicability of the Kac-Rice formula while bounding the expected difference in critical points of the field and its Palm version. 

\item  In the proof, we will mostly use the following properties of the field rather than the explicit covariance structure for most of the computation.

    \begin{enumerate}
        \item $f$ is a unit variance, zero mean, stationary field, isotropic.
        \item (Non-degeneracy) The Gaussian vector   
                    \[
                        (f(x), \nabla f(x), \hess f(x))
                     \]
                         non-degenerate. See \cite[Appendix A]{beliaev_smoothness_2020}.
                    
        \item Almost surely, $f$ is a $C^4$-smooth function (it is in fact analytic).  
        
\end{enumerate}

\end{itemize}

The main idea of the proof is an important characterisation Poisson point process: A point process $\zeta$ is a Poisson process if and only if $\zeta$ conditioned to have a point at $x$ has the same distribution as $(\zeta+ \delta_x)$, for all $x$. Following Stein's idea to approximate a Gaussian distribution by looking at functionals, Chen-Xia \cite{chen_steins_2004} showed that if the functionals of the Palm measure of $\zeta$ at $x$ and that of $(\zeta+\delta_x)$ are close, then $\zeta$ is close to being a Poisson point process.
We will exploit this idea to quantify Poisson point process convergence of high critical points of stationary Gaussian fields.

\subsection{Level above the expected maxima}

 The excursion of the field above level $u(n)= \sqrt{2\alpha d \log n}$ for some $\alpha>1$ in an $n \times n$ box is empty with high probability because the maxima of stationary smooth Gaussian fields is exponentially concentrated around the mean \cite{tanguy_superconcentration_2015}. In this case, we can strengthen the result to show that the discrete excursion set is close the identically zero process, in total variation distance.

\pgap

\begin{assumptions} \label{assumptions-3}
Consider a $C^2$-smooth Gaussian field $f: \bR^d \to \bR $, with covariance kernel $r$ which satisfies the following conditions. 
    \begin{enumerate}
        \item $f$ is stationary, unit variance, zero mean. 
        \item $r(t)= o(1)$  as $t \to \infty$ (Correlation decay).
    \end{enumerate}
    Note that $f$ being stationary and $r(t) \to 0$ as $t \to \infty$ implies that $r(s-t)\neq 1$ for all $s \neq t$.
\end{assumptions}

\begin{theorem} \label{thm-supercritical-level}
 Let $f: \bR^d \to \bR$ be a smooth Gaussian field satisfying Assumptions \ref{assumptions-3}. Let $I_n=[0,n]^d \cap \bZ^d, D=[-1/2,1/2]^d, D_t= t+D$ for $t \in I_n$. For $u(n)=u>0$, let 
\[
X_{t,n}=X_t=\bI \left [\max_{x \in D_t} f(x) > u \right ] .
\]
Let $X' \equiv 0$ be the zero process. For fixed $\alpha >1$, let 
\[
u(n)=\sqrt{2 \alpha d \log n}.
\]
Then, for some constant $c_0>0$
 \[
 \dfrac{\|\CL(X)-\CL(X')\|_{TV}}{n^{(1-\alpha)d}(\log n)^{(d-1)/2}}  \to c_0 \quad \text{as} \quad n \to \infty.
 \]
\end{theorem}

\pgap 

Remarks: Using more quantitative version of Theorem \ref{thm-piterbarg-excursion} (i.e. results from \cite[Chapter 14]{adler_random_2009}) we can show that, for some fields like Bargmann-Fock field, 
\[
 \left |c_0-\dfrac{\|\CL(X)-\CL(X')\|_{TV}}{n^{(1-\alpha)d}(\log n)^{(d-1)/2}} \right | = O \left (\dfrac{1}{\sqrt{\log n}} \right ) = O(1/u)  \quad \text{as} \quad n \to \infty.
 \]

Also, we get similar upper bounds as in Theorem \ref{thm-supercritical-level} if we compare the process $X_t$ with a Bernoulli $X'$ with same marginals (i.e. $X_t \overset{d}{=}X'_t$) and independent entries $X'_s$ using the Chen-Stein's method as in \cite[Theorem 2.1]{sousi_chenstein_2020}. 

\section{Proofs}

\subsection{Proof of Theorem \ref{thm-weak-conv}}

It is well known at least since the 1970's that avoidance probabilities (i.e. $\bP(\eta(B)=0)$ for Borel sets $B$) characterise simple point process (i.e. point processes with mass concentrated only on atoms where all delta measures have equal weights). 
Now, weak convergence of these point processes can be studied by scrutinising avoidance probabilities and intensity measures. 

\pgap

\begin{definition}[DC-ring]
  Let $\CB$ be the Borel $\sigma$-algebra on $\bR^d$. A ring $\mathfrak{L} \subset \CB $ is called a \textit{ DC-ring} (`dissecting covering' ring) if for any compact set $K$ from $\CB$, and arbitrary $\ep>0$, there exists a finite covering of $K$ by some sets $l \in \mathfrak{L}$ such that diam $l \leq \ep$.  
\end{definition}

Let $\mathfrak{L}$ be a ring generated by rectangles 
\[
\prod_{i=1}^{d}[t_i,t_i+s_i), \quad s_i\geq 0, i=1,2, \ldots, d
\]
which will be a DC-ring with the property that $\Phi(\partial l)=0$ a.s. for any $l \in \mathfrak{L}$.

\begin{theorem}[c.f. Theorem 4.18 of \cite{kallenberg_random_2017}] \label{thm-kallenberg}
    If 
    \begin{equation} \label{eq:to-show}
\begin{aligned}
    &\lim_{R \to \infty} \bP(\Phi_u(l)=0)= \bP(\Phi(l)=0),
    \\
    & \limsup_{R \to \infty} \bE \Phi_u(l) \leq \bE \Phi (l)
\end{aligned}
\end{equation}
for all $l \in \mathfrak{L}$, then we have 
\[\Phi_u \to \Phi \quad \text{ as } u \to \infty\]
in law. 
\end{theorem}

We will show \eqref{eq:to-show} for $L =[0,1]^d$ but the argument works for any $l \in \mathfrak{L}$. 

\pgap

First, we approximate avoidance probabilities of the sequence $\Phi_u$ by the non-exceedance probabilities of the field $f$ (Lemma \ref{lemma:avoidance-approximation}). We then approximate the non-exceedance probabilities on rectangles $\bP( \sup_{ \mu(u) \cdot L}f > u )$  by that on a grid which is fine enough (Lemma \ref{lemma:discretise}). 
Then we show that for a regular enough field $f$ with unit variance, the excursion set $\{f>u\}$ is captured by a grid with spacing of order $u^{-1}$ for large $u$. 
Now we compare the non-exceedance probabilities of the field $f$ to that of the field $f_0$ which is an i.i.d copy of $f$ on each fixed box. This is done by the comparison method for Gaussian vectors \cite[Theorem 1.1]{piterbarg_asymptotic_1996} and is the same as the proof of \cite[Theorem 15.2]{piterbarg_asymptotic_1996}.
Lastly, from Lemma \ref{lemma:poisson-conv} we show that non-exceedance probabilities of the field $f_0$ converge to the avoidance probabilities of the Poisson point process, which proves the first part of  \eqref{eq:to-show}.

\pgap 

The second part of \eqref{eq:to-show} deals with the expected number of critical points of a given index of smooth Gaussian fields. It is a classical problem in Gaussian analysis (see \cite{adler_geometry_2010}). Thanks to Kac-Rice formulas, we know precise estimates of these quantities, even explicit results in some cases. 
Using these estimates, we will show that 
\[
\lim_{u \to \infty} \bE [\Phi_{u}(L)] = \bE[\Phi(L)].
\]

One of the key estimates in the proof of Theorem \ref{thm-weak-conv} is the asymptotic excursion probability of stationary smooth Gaussian fields. 

\begin{theorem}[Theorem 7.1,\cite{piterbarg_asymptotic_1996}] \label{thm-piterbarg-excursion}
  Let $X: \bR^d \to \bR$ be a zero mean, unit variance, stationary $C^2$-smooth Gaussian field. Further assume that $(X(s), \nabla X(s))$ is a non-degenerate Gaussian vector.  Let $r(t,s)$ be the covariance function of the field $X$ such that $r(t,s)<1$ for $t\neq s$. Let $A \subset \bR^d$ be a Jordan set of positive measure. Then, 
    \[
    \bP \left ( \max_{t \in A} X(t) >u \right )= C \vol (A) u^{d-1} \Psi(u) (1+o(1)) \qquad {\rm{ as }} \quad u \to \infty. 
    \]
    Here, the constant $C$ depends only on the field and not on level $u$, $1-\Psi$ is the CDF of a standard Gaussian. 
\end{theorem}

Recall that $L$ is a unit box in $\bR^d$ and let $L_u:= \mu(u) \cdot L$. Define
\[
P_f(u,S):= \bP \left ( \sup_{t \in S} f(t) \leq u \right ) \quad \text{ and } \quad \overline{P}_f(u,S):= \bP \left ( \sup_{t \in S} f(t) \geq u \right ). 
\]

Define $\widetilde{L_u}:= \{x: \text{ dist}(\partial L_u,x) <1\}$, i.e. $1-$neighbourhood of the boundary of $L_u$. Now we approximate the avoidance probability of the point process with non-exceedance probabilities. 
\begin{lemma} \label{lemma:avoidance-approximation}
    With the above setup, we have 
    \[
    \bP(\Phi_u(L)=0) = P_{f}(u,L_u)+o(1) \quad \text{ as } u \to \infty . 
    \]
\end{lemma}

\begin{proof}
    First, observe that $\bP(\Phi_u(L)=0) \geq P_{f}(u,L_u)$. 
    From the fact that each connected component of $\{f(x) \geq u\}$ must have a local maximum, we have 
    \[
    \{\Phi_u(L)> 0\} \supseteq \left \{\sup_{L_u}f \geq u, \sup_{\widetilde{L_u} \setminus L_u} f < u\right \}.
    \]
Note that the RHS ensures that $L_u$ has at least one component of $\{f(x) \geq u\}$ lying completely inside it. Although the boundary $\partial L_u$ suffices instead of $\widetilde{L_u} \setminus L_u$, we need a ``thickened" version  of the boundary (i.e. a positive measure in the ambient dimension) to directly apply Theorem \ref{thm-piterbarg-excursion}. Hence, 
\[
\bP(\Phi_u(L)=0) \leq P_{f}(u, L_u) + \bP \left (\sup_{L_u}f \geq u, \sup_{\widetilde{L_u} \setminus L_u} f \geq u \right ).
\]

Now, 
\[
\bP \left (\sup_{L_u}f \geq u, \sup_{\widetilde{L_u} \setminus L_u} f \geq u \right ) \leq \overline{P}_{f}(u, \widetilde{L_u} \setminus L_u ).
\]

Note that $\text{vol}(\widetilde{L_u} \setminus L_u)= O(\mu(u)^{d-1})$ for large $u$. Applying Theorem \ref{thm-piterbarg-excursion}, using stationarity of the field,  we have 
\begin{equation*}
\begin{split}
\overline{P}_{f}(u, \widetilde{L_u} \setminus L_u ) \leq &C \cdot \text{vol}(\widetilde{L_u} \setminus L_u ) u^{d-1} \exp(-u^2/2)\\
                                           =& O(\mu(u)^{-1}) \text{ as } u \to \infty. 
\end{split}
\end{equation*}
    
\end{proof}

As explained, we discretise the domain and approximate the non-exceedance probabilities on this grid. 
Fixing $b>0$, define $\CR_{b,u}= b u^{-1}\bZ^d$. 

\begin{lemma} \label{lemma:discretise}
    For any $\ep>0$, there exists $b,u_0>0$ such that for all $u>u_0$, 
    \[
    P_{f}(u,L_u \cap \CR_{b,u}) - P_{f}(u,L_u) \leq \ep. 
    \]
\end{lemma}

\begin{proof}
    We have 
    \[
    P_{f}(u,L_u \cap \CR_{b,u}) - P_{f}(u,L_u) = \bP \left ( \sup_{L_u \cap \CR_{b,u}} f \leq u, \sup_{L_u}f > u \right ).
    \]
   Let $T:= \lceil \mu(u) \rceil $, divide the $[0,T]^d$ into unit boxes. By the union bound and the stationarity of the field $f$, 
\begin{equation*}
    \begin{split}
        \bP \left (\sup_{L_u \cap \CR_{b,u}} f \leq u, \sup_{L_u}f > u \right ) & \leq \bP \left (\sup_{T \cap \CR_{b,u}} f \leq u, \sup_{T}f > u \right )\\
        & \leq c \mu(u)^d \bP \left ( \sup_{[0,1]^d \cap \CR_{b,u}} f \leq u, \sup_{[0,1]^d} f >u \right ). 
    \end{split}
\end{equation*}

Divide the cube $[0,1]^d$ into smaller cubes congruent to $[0,u^{-1}]^d$. By the union bound and the stationarity of the field $f$, 
\begin{equation} \label{eq:discrete-approx-1}
    \bP \left (\sup_{L_u \cap \CR_{b,u}} f \leq u, \sup_{L_u}f > u \right ) \leq c u^d \mu(u)^d \bP \left ( \sup_{[0,u^{-1}]^d \cap \CR_{b,u} }  f \leq u, \sup_{[0,u^{-1}]^d} f >u \right ).
\end{equation}

From Appendix \ref{sec: discrete-approx-appendix}, specifically \eqref{eq:discrete-approx} with $M=[0,1]^d$ so that $u^{-1}M=[0,u^{-1}]^d$, we have 
\begin{multline} \label{eq:discrete-approx-2}
   \lim_{ u \to \infty} \dfrac{1}{\sqrt{2 \pi}} u e^{u^2/2} \bP \left ( \sup_{[0,u^{-1}]^d \cap \CR_{b,u} }  f \leq u, \sup_{[0,u^{-1}]^d} f >u \right ) = \\ 
   \int_0^\infty e^v \bP \left ( \sup_{[0,1]^d \cap (b \bZ^d)} (\chi(t) - \|t\|^2) \leq  v, \sup_{[0,1]^d}( \chi(t) - \|t\|^2) > v \right ) dv
\end{multline}
where $\chi$ is a continuous Gaussian field defined in intro of Appendix \ref{sec: discrete-approx-appendix}. As $b \to 0$, the integrand of the RHS of \eqref{eq:discrete-approx-2} tends to zero sample paths of $\chi$ are continuous a.s. Hence, by the dominated convergence theorem, the integral itself tends to zero as $b \to 0$. Note that, since
\[ u^d \mu(u)^d = u^{-1} e^{-u^2/2},\]
 the RHS of \eqref{eq:discrete-approx-1} bounded by RHS of \eqref{eq:discrete-approx-2} (up to a constant factor) when $u$ is large enough. So choosing $b>0$ small enough proves the lemma. 

\end{proof}

 Divide the rectangle $ L_u $ into smaller ones by following construction. 
Divide each side of $ L_u $ into segments of length `$a$' alternated by that of $\delta$. Define $\lambda_{a,u}$ to be the union of cubes of side length $a$. 
Note that the distance between the cubes is greater than $\delta$. The following lemma says that if the gap between the cubes of $\lambda_{a, u}$ are small enough, then the non-exceedance probabilities of the discretisation of $ L_u$ and of $\lambda_{a,u}$ are close. 

\begin{figure}[h]
    \centering
    \includegraphics[width=0.8\linewidth]{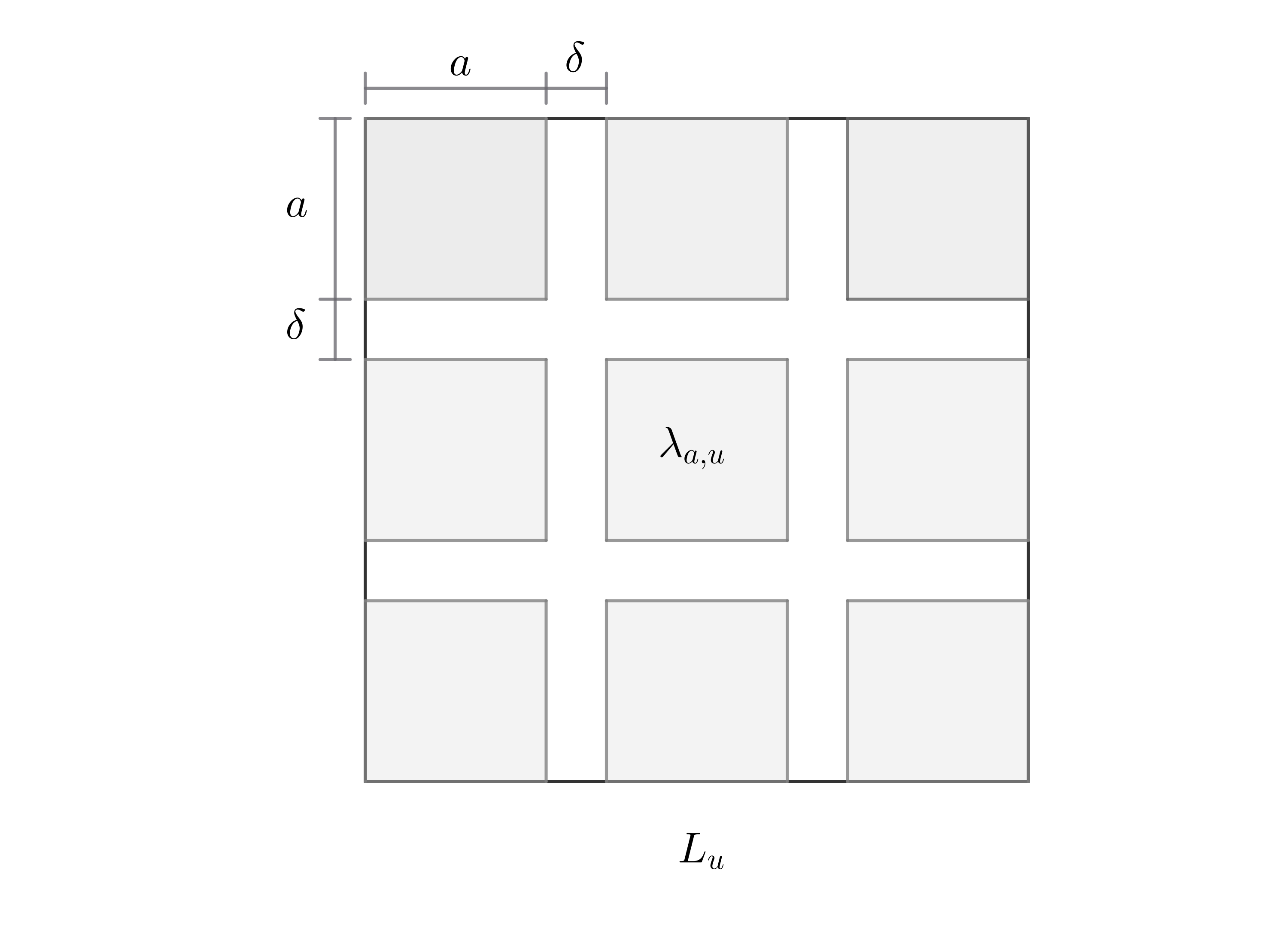}
    \caption{The entire shaded area is defined to be $\lambda_{a,u}$}
    \label{fig:lambda}
\end{figure}

\begin{lemma}
    For any $a, \ep >0$ given, there exists $\delta> 0$, such that, for all $u$ large enough we have, 
    \[
    P_f(u, \lambda_{a,u} \cap \CR_{b,u})- P_f(u,  L_u \cap \CR_{b,u}) \leq \ep. 
    \]
\end{lemma}

\begin{proof}
    We have that 
    \[
    P_f(u, \lambda_{a,u} \cap \CR_{b,u})- P_f(u,  L_u \cap \CR_{b,u})  \leq \bP \left ( \sup_{\lambda_{a,u} \cap \CR_{b,u}} f \leq u, \sup_{ L_u \cap \CR_{b,u}} f > u \right ).
    \]
    Now using the stationarity of the field, 
    \begin{equation*}
        \begin{split}
        \bP \left ( \sup_{\lambda_{a,u} \cap \CR_{b,u}} f \leq u, \sup_{ L_u \cap \CR_{b,u}} f > u \right )  & \leq \overline{P}_f(u, L_u \setminus \lambda_{a,u}  ) \\
          & \leq \vol( L_u \setminus \lambda_{a,u}) \overline{P}_f(u, L) \\
          & \leq C_1 \delta \dfrac{(\mu(u))^d}{(a+\delta)} \overline{P}_f(u, L) \\
          & \leq C_2 \delta ((\mu(u)R)^d) u^{d-1} \exp(-u^2/2) \\
          & \leq C_3 \delta 
        \end{split}
    \end{equation*}
    where $C_3$ is a constant which does not depend on $u$. 
\end{proof}

Let $f_0$ be a field defined on $\lambda_{a,u}$ such that on the cubes of side length $a$, the field is made up of independent copies of $f$. We now show that the non-exceedance probability of $f_0$ converges to the avoidance probability of the Poisson point process in $L$.

\begin{lemma} \label{lemma:poisson-conv} We have
    \[
    P_{f_0}(u, \lambda_{a,u}) \to \exp (- \vol(L)) \quad \text{ as } u \to \infty.
    \]
\end{lemma}

\begin{proof}
    Let $N$ be the number of cubes of side length $a$ in $\lambda_{a,R}$. Then, 
    \[
    P_{f_0}(u, \lambda_{a,u})= (1-\overline{P}_f(u,[0,a]^d))^N
    \]
    by the independence of the field on these cubes. Taking the logarithm, it is enough to estimate
    \[
    N \log(1-\overline{P}_f(u,[0,a]^d))= -N \overline{P}_f(u,[0,a]^d) + O(N\overline{P}_f(u,[0,a]^d)^2). 
    \]

    Now by Theorem \ref{thm-adler-density}, 
    \[
    \overline{P}_f(u,[0,a]^d)= a^d \mu(u)^{-d}(1+O(u^{-1}))
    \]
    and 
    \[
    N= \left ( \dfrac{ \mu(u)}{a+\delta} \right )^d +O(( \mu(u))^{d-1}).
    \]
    Hence, 
    \[
    N \overline{P}_f(u,[0,a]^d) = \left (\dfrac{a}{a+\delta} \right )^d  +o(1) \text{ and } N \overline{P}_f(u,[0,a]^d)^2 =o(1).
    \]

    We have the result since $L$ is a unit box and we can take $\delta$ arbitrarily small. 
\end{proof}

The following lemma would imply the second condition in \eqref{eq:to-show}.

\begin{lemma} \label{lemma-expectation-computation} We have,
\[
\lim_{R \to \infty} \bE [\Phi_{u}(L)] = \bE[\Phi(L)].
\]

\end{lemma}

\begin{proof}
    Let us first consider the case that $(f(0), \nabla f(0), \hess f(0))$ has a  density in $\bR^{(d+1)+d(d+1)/2}$. 
We follow the convention that we consider the upper triangle part of the matrix $\nabla ^2 f(0)$ in the vector $(f(0), \nabla f(0), \hess f(0))$. As mentioned, we will use the following theorem by Adler

\begin{theorem}[ \cite{adler_geometry_2010} Theorem 6.3.1] \label{thm-adler-density}
    Let $f: \bR^d \to \bR$ be a stationary, $C^2$-smooth Gaussian field such that $(f(x),\nabla f(x), \nabla ^2 f(x))$ is non-degenerate for all $x \in \bR^d$. Further, assume that $f(x)$ has zero mean and unit variance. 
    Let $M_u(f,S)$ denote the number of local maxima of $f$ in $S \subset \bR^d$ with $f>u$. Then, 
    \[
    \bE[M_u(f,S)]=\dfrac{\vol (S)\det(\Lambda_f)^{1/2}u^{d-1}}{(2 \pi)^{(d+1)/2}} \exp{(-u^2/2)}(1+O(u^{-1}))
    \]
    where $\Lambda_f$ is the covariance matrix of $\nabla f$ and $O(u^{-1})$ is independent of choice of $S$.
\end{theorem}

For the case that $(f(0), \nabla f(0), \hess f(0))$ has a  density in $\bR^{(d+1)+d(d+1)/2}$, \ref{thm-adler-density} suffices.
If the field is isotropic and if $(f(0), \nabla f(0), \hess f(0))$ is degenerate (where the upper triangle of the Hessian of $f$ is vectorised), then the field has to be a monochromatic random wave (MRW) (see \cite[Prop 3.10]{cheng_expected_2018}). 
From Example 3.15 of \cite{cheng_expected_2018}, we can calculate the limit of $\bE[\Phi_{u}(L)]$ for the case $d=2$. 
But explicit expressions for height densities are hard to get for $d\geq 3$ directly. 
So we shift the MRW field by an independent normal random variable, so that the joint vector of the field, its gradient, and Hessian has a density.
Then we use the explicit asymptotic as in \ref{thm-adler-density}.

\pgap 

By Theorem \ref{thm-adler-density}, for any Borel set $B \subset \bR^d$ 

\begin{equation} \label{eq: expectation limit}
    \begin{split}
        \bE[\Phi_u(B)] & = \bE[M_u(f, \mu(u)B)] \\ 
         & = \dfrac{\vol ( \mu(u)  B)}{(2\pi)^{(d+1)/2}} u^{d-1} \exp(-u^2/2)(1+O(u^{-1})) \\
         & = \vol(B) (1+O(u^{-1})) \\
         & \to \bE[\Phi(B)] \quad \text{ as } u \to \infty.
    \end{split}
\end{equation}

Here, we have used the fact that the determinant of the covariance matrix of $\nabla f$ is 1, which follows from point 4 of Assumption \ref{assumptions-2}.

\pgap

Now we consider the monochromatic random waves (MRW) case. Let $f: \bR^d \to \bR$ be an MRW field. Let $\ep >0$ and consider a random variable $N$, independent of the field $f$, which is a standard Gaussian r.v. 
Define, 
\[
f_{\ep}(x):= f(x)+ \ep N , \quad x \in \bR^d.
\]

Observe that $f_{\ep}$ is still a centered, stationary field and that $f_{\ep}(0), \nabla f_{\ep}(0), \hess f_{\ep}(0)$ is a Gaussian vector with density. 
Define $M_u(g)$ to be the number of local maxima of a Gaussian field $g$ in $[0,1]^d$ above level $u$. 

Now we have, by an application of the Kac-Rice formula,
\[
\bE[M_u (f_{\ep})]= \int_{\bR} \bE[M_{u-\ep b}( f) | N=b] \phi(b) d \rm{b}.
\]
where $\phi$ is the probability density of a standard normal variate. Also, 

\[
M_{u - \ep b} (f) \longrightarrow M_u( f) \quad \rm{a.s.} \quad \rm{as} \quad \ep \to 0.
\]

Note that $M_u(f)$ is integrable and monotonic with respect to $u$, given that $u$ is large, so using the dominated convergence theorem, 
\[
\bE[M_{u - \ep b} (f)] \to \bE[M_u(f)] , \quad \ep \to 0.
\]

Since $\bE[M_u(f)]$ is uniformly bounded in $u$, apply the dominated convergence theorem for $\bE[M_{u-\ep b}( f)] \phi(b)$ to get, 
\[
\lim_{\ep \to 0} \bE[M_{u}(f_{\ep})]= \bE[M_u( f)].
\]

Computing $\bE[M_u(f_{\ep})]$ is the same as \eqref{eq: expectation limit} by applying Theorem \ref{thm-adler-density}.

\end{proof}

\begin{proof}[Proof of Theorem \ref{thm-weak-conv}]
    First, observe that all the proofs of Lemmas \ref{lemma:avoidance-approximation} to \ref{lemma:poisson-conv} goes through even when $L$ is a finite union of finite rectangles. For any given $\ep >0$, there exists $a,b,\delta, u_0$ such that for all $u>u_0$, 
    \[
    |\bP(\Phi_u(L)=0)- P_f(u, \lambda_{a,u} \cap \CR_{b,u})| \leq \ep. 
    \]

    If we show that $|P_f(u, \lambda_{a,u} \cap \CR_{b,u})-P_{f_0}(u, \lambda_{a,u} \cap \CR_{b,u})| \to 0$ as $u \to \infty$ then, together with Lemma \ref{lemma-expectation-computation} and Theorem \ref{thm-kallenberg} implies the conclusion of Theorem \ref{thm-weak-conv}.

    \pgap 

    Let $K_i$ be a renumbering of cubes with edges of length $a$ which comprise $\lambda_{a,u}$, $i=1,2,\ldots, N$. Let covariance of the field $f_0$ on $\lambda_{a,u}$ be denoted by $r_0(t,s)$. Define $\lambda'_{a,u,b}=\lambda_{a,u} \cap \CR_{b,u}$. Then by Theorem \ref{thm-method-of-comparision}, we have 

    \begin{equation}
    \begin{split}
        |P_f(u, \lambda_{a,u,b}')- & P_{f_0}(u, \lambda_{a,u,b}')| \leq \frac{1}{\pi} \sum_{t,s \in \lambda'_{a,u,b}} |r(t-s)-r_0(t-s)| \\
        & \times \int_0^1 (1-(hr(t,s))^2)^{-1/2} \exp \left (-\dfrac{u^2}{1+hr(t,s)} \right ) dh.
    \end{split}    
    \end{equation}
Denote the summand on the RHS of the above equation by $\beta(t,s)$. If $t,s \in K_i$ for some $i$, then $r(t,s)=r_0(t,s)$, hence $\beta(t,s)=0$.

Now consider the case that $t,s$ belong to different $K_i$ and $K_j$ such that $|t-s|\leq \mu(u)^{\gamma_1}$, where $\gamma_1>0$ is a constant chosen later. 
Since $t,s$ belong different cubes, we have $|t-s|>\delta$, hence $|1-r(t,s)|> \gamma_2>0$. So, 
\[
\dfrac{1}{1+r(t,s)} > \dfrac{1}{2} + \dfrac{\gamma_2}{4}.
\]

Now, 
\begin{equation}
    \begin{split}
        \sum_{ \substack{t \in K_i, s \in K_j, i\neq j, \\ |t-s|< \mu(u)^{\gamma_1}}} \beta(t,s) & \leq  C_1 \sum |r(t,s)| \exp \left ( -\dfrac{u^2}{1+r(t,s)} \right ) \\
        & \leq C_2 (\mu(u))^d \mu(u)^{\gamma_1 d} \exp (-(1+\gamma_2/2)u^2/2) \\
        & \leq C_3 (u^{1-d}\exp(u^2/2))^{1+\gamma_1}  \exp (-(1+\gamma_2/2)u^2/2) \\
        & \rightarrow 0 \quad {\rm{ as }} \quad  u \to \infty \quad {\rm{ if }}  \quad 0<\gamma_1 < \gamma_2/2.
    \end{split}
\end{equation}

 $C_i$'s are different constants not depending on $u$.

\pgap 

Lastly, we consider the case where $|t-s|\geq \mu(u)^{\gamma_1}$. We have, 
\begin{equation}
    \begin{split}
        \sum_{ \substack{t \in K_i, s \in K_j, i\neq j, \\ |t-s| \geq \mu(u)^{\gamma_1}}} \beta(t,s) & \leq C_1 \sum |r(t,s)| \exp \left ( -\dfrac{u^2}{1+r(t,s)} \right ) \\
        & \leq C_2 (\mu(u))^{2d} r'(\mu(u)^{\gamma_1}) \exp \left ( -\dfrac{u^2}{1+r'(\mu(u)^{\gamma_1})} \right ) \\
        & \leq C_3 u^{2-2d} r'(\mu(u)^{\gamma_1}) \exp \left ( \dfrac{r'(\mu(u)^{\gamma_1}) u^2}{1+r'(\mu(u)^{\gamma_1})} \right )
    \end{split}
\end{equation}
where 
\[
r'(h):= \max_{|t|\geq h} |r(t,0)|, \quad h \in (0, \infty).
\]

Observing that the assumption on the decay of correlation (point 2 of Assumption \ref{assumptions-2}) implies that $u^{2}r'(\mu(u)^{\gamma_1}) \to 0$ as $u \to \infty$. This also implies $u^{2-2d}r'(\mu(u)^{\gamma_1}) \to 0$.

Hence, 
\[
\sum_{ \substack{t \in K_i, s \in K_j, i\neq j, \\ |t-s| \geq \mu(u)^{\gamma_1}}} \beta(t,s) \rightarrow 0 \quad {\rm{as}} \quad u \to \infty.
\]
This completes the proof of Theorem \ref{thm-weak-conv}.
\end{proof}

\subsection{Proof of Theorem \ref{thm-number-count}}

\textbf{Poisson approximation in Wasserstein distance}: 
Denote by $\CH$ the space of non-negative integer-valued locally finite measures on a domain $\Gamma \subset \bR^2$. For a point process $\zeta$ denote by $\zeta^{\alpha}$ a Palm version of $\zeta$ at $\alpha \in \Gamma$ (see appendix \ref{appendix-palm} for definition of Palm measures). Palm conditioning is the ``horizontal" condition (i.e. $\nabla f(x)=0$ with $|x| < \epsilon$ where $\ep \to 0$), as opposed to the Gaussian regression conditioning (i.e. $|\nabla f(0)| < \ep$ where $\ep \to 0$). If $0 \in \Gamma$, denote $\zeta^0 = \Tilde{\zeta}$ for the Palm version of $\zeta$ at the origin. Later, we will adapt the same notation for Palm version of a Gaussian field $f$ conditioned to have a local maxima at $0$.

\pgap 

Let us define a pseudo-metric on $\CH$ and then study the 1-Wasserstein distance on probability measures on $\CH$ with respect to this pseudo-metric. Let $\rho_0$ be any bounded pseudo-metric on $\CH$. For two simple point configurations $\xi_1, \xi_2 \in \CH$ such that 
\[
\xi_1 = \sum_{i=1}^n \delta_{x_i}, \quad \xi_2= \sum_{i=1}^m \delta_{z_i}
\]
with $m \geq n$, define 
\[
\rho_1(\xi_1, \xi_2)= \min_{\pi} \sum_{i=1}^n \rho_0(x_i, z_{\pi(i)}) + (m-n). 
\]
where $\pi$ runs over $S_n$, symmetric group over $n$ elements.

\pgap 

Let $\CF$ denote the set of $1$-Lipschitz functions with respect to $\rho_1$ on $\CH$. Then the 1-Wasserstein metric on probability measures on $\CH$ is given by, 
\[
\rho_2(Q_1,Q_2)= \sup_{f \in \CF} \left | \int f d Q_1 -\int f dQ_2 \right |
\]

Let $\CL \Phi_{R,u}$ denote the law of the point process $\Psi_{R,u}$ and $\mathcal{U}_{R,u}$ be the Poisson point process with intensity measure being the Lebesgue measure on $D_{R,u}$. Note that when the pseudo-metric $\rho_0 \equiv 0$
\[
\rho_2(\CL \Psi_{R,u}, \mathcal{U}_{R,u} )= d_{TV}(|\Psi_{R,u}|,|\mathcal{U}_{R,u}|).
\]

Another interesting choice for the pseudo-metric $\rho_0$ is 
\[
\rho_0(x,y)= \min(\|x-y\|,1).
\]
In this case, $\rho_2(\CL \Psi_u, \mathcal{U}_{R,u})$ measures how `far' we have to move the points of $\Phi_u$, on average, to make it look like Poisson point process.  In this case, we suspect that  the upper bound we will have extra factor $\vol(D_{R,u})$. 

\paragraph{Chen-Stein Approximation using Palm measures:} Let $\Xi$ be a point process on a locally compact second countable Hausdorff space $\Gamma$ with finite mean measure $\lambda$, and let $A_{\alpha}\subseteq\Gamma$ be measurable neighbourhoods indexed by $\alpha\in\Gamma$.  Denote by $\Xi(\Gamma)$ the total number of points of $\Xi$ in $\Gamma$.  Chen and Xia \cite{chen_steins_2004} obtained a total variation bound between the law of $\Xi(\Gamma)$ and a Poisson random variable with mean $\lambda = \lambda(\Gamma)$.  Their result is as follows.

\begin{theorem}[\cite{chen_steins_2004}, Theorem 3.1] \label{thm-chen-xia-tv}
Let $\mathrm{Pois}(\lambda)$ denote a Poisson random variable with mean $\lambda$, and write $d_{\mathrm{TV}}$ for the total variation metric.  Then
\begin{multline}
 d_{\mathrm{TV}\!}\bigl(\mathcal{L}(\Xi(\Gamma)), \mathrm{Pois}(\lambda)\bigr)
 \leq \frac{1 - e^{-\lambda}}{\lambda}
   \mathbb{E}\int_{\alpha\in\Gamma} \bigl(\Xi(A_{\alpha}) - 1\bigr)\,\Xi(d\alpha)
   \\ + \min\{\varepsilon_{1},\varepsilon_{2}\}
   + \frac{1 - e^{-\lambda}}{\lambda}
   \int_{\alpha\in\Gamma} \lambda(A_{\alpha})\,\lambda(d\alpha),
\end{multline}
where
\begin{align*}
 \varepsilon_{1}
 &= \bigl(1\wedge \sqrt{2e/\lambda}\bigr)
    \int_{\alpha\in\Gamma}
      \mathbb{E}\bigl\lvert G(\alpha,\Xi(\alpha)) - \varphi(\alpha)\bigr\rvert\,\nu(d\alpha),\\
 \varepsilon_{2}
 &= \frac{1 - e^{-\lambda}}{\lambda}
    \int_{\alpha\in\Gamma}
      \mathbb{E}\Bigl|\,\bigl|\Xi(\alpha)\bigr| - \bigl|\bigl(\Xi^{\alpha}\bigr)(\alpha)\bigr|\Bigr|\,\lambda(d\alpha).
\end{align*}
Here $G$ and $\varphi$ are functions appearing in the Stein equation for the Poisson process, $\nu$ is a reference measure, and $\Xi^{\alpha}$ denotes the Palm version of $\Xi$ at $\alpha$.  The bound for $\varepsilon_{1}$ holds when $\Xi$ is simple.\qedhere
\end{theorem}

\pgap 

For each $x \in D_{R,u}$ let $A_{x}$ denote the open ball of radius $\tau(u)$ in $D_{R,u}$ around $x$. Let $V_{R,u}:= \vol(D_{R,u})$, which is the expected total mass of $\Psi_{R,u}$. 
Recall that $\Psi^x$ denote the Palm version of a point process $\Phi$ at $x$. Applying Theorem \ref{thm-chen-xia-tv} with $\lambda= V_{R,u} $, we have
\begin{equation}
\begin{split}
    d_{TV}(\CL |\Psi_{R,u}|, U_{R,u}) \leq \dfrac{1- e^{-V_{R,u}}}{V_{R,u}}  \left ( \bE  \int_{D_{R,u}} (\Psi_{R,u}(A_x)-1) \Psi_{R,u}(\td x) + \right .\\
    \left . \int_{D_{R,u}} \bE||\Psi_{R,u}^x(A_x^c)|-|\Psi_{R,u}(A_x^c)|| \td x + \int_{D_{R,u}} \vol(A_x) \td x \right )
    \end{split}
\end{equation}

Notice that only the second term of the RHS above depends on the coupling of the field $f$ and its Palm version $\tf$. Using the stationarity of the field $f$, we have,
\begin{equation} \label{eqn:chen-stein-bound}
\begin{split}
    d_{TV}(\CL |\Psi_{R,u}|, U_{R,u}) \leq   \left ( \dfrac{1}{V_{R,u}}\bE  \int_{D_{R,u}} (\Psi_{R,u}(A_x)-1) \Psi_{R,u}(\td x) + \right .\\
    \left .  \bE||\Psi_{R,u}^0(A_0^c)|-|\Psi_{R,u}(A_0^c)|| + \vol(A_0) \right )
    \end{split}
\end{equation}

We will choose the radius parameter $\tau(u)$ in the course of the proof.

\pgap

There are three main components of the proof. The first term in the RHS above is the expected number of `cluster' points of the process $\Psi_{R,u}$ (i.e. points of $\Psi_{R,u}$ which are within distance $\tau(u)$ of each other). The second component is to come up with a coupling of the field $f$ and its Palm version $\tf$ such that $\|f(x)-\tf(x)\|_{C^4}$ is small for $ x \gg 1$. The third component is the bound on the difference of the number of high critical points of these two fields in $L^1$.

\subsubsection{Bounding the number of cluster points}

We have
\begin{equation*}
    \begin{split}
        \int_{D_{R,u}} (\Psi_{R,u}(A_{x})-1)\Psi_{R,u}( \td x)= \sum_{x \in \Psi_{R,u}} \Psi_{R,u}(A_{x})-1.
    \end{split}
\end{equation*}

For a box $M$, if $\Psi_{R,u}$ has $k$ points in the $\tau(u)$-neighborhood of the box $M$, then
\[
\int_{M} (\Psi_{R,u}(A_{x})-1)\Psi_{R,u}( \td x) \leq k(k-1).
\]
The reasoning is that each point in $M$ can have at most $(k-1)$ points within distance $\tau(u)$, and there are at most $k$ points in $M$. Hence, by linearity of expectation and stationarity of the field $f$, 
\begin{equation}
    \begin{split}
        \bE \int_{D_{R,u}} (\Psi_{R,u}(A_{x})-1)\Psi_{R,u}( \td x) \leq  \dfrac{V_{R,u}}{\tau(u)^2} \bE[\eta(\eta-1)]
    \end{split}
\end{equation}
where 
\[
\eta= \text{ number of local maxima of } f \text{ above level } u \text{ in a square of side }  \tau(u) \mu(u). 
\]
Here, $V_{R,u}/\tau(u)^2$ is the number of squares of side $ \tau(u) \mu(u)$ that fit in $R \times R$ box. 

\pgap

We now compute $\bE[\eta(\eta-1)]$ using the Kac-Rice formula. We have, from the second factorial moment of the Kac-Rice formula \cite[Theorem 6.3]{azais_level_2009}, denoting $g_u= \tau(u) \mu(u)$,
\begin{equation} 
    \begin{split}
        \bE[\eta(\eta-1)]= \int_{[-g_u/2,g_u/2]^2} \int_{[-g_u/2,g_u/2]^2} \bE \left[|\det \hess f(x)||\det \hess f(y)| \right . \\
        \left . \bI[f(x) \geq u, f(y)\geq u] \big | \nabla f(x)=0=\nabla f(y) \right] \psi_{x,y}(0,0)\td x \td y
    \end{split}
\end{equation}
where $\psi_{x,y}$ is the pdf of $(\nabla f(x), \nabla f(y))$. By stationarity of the field $f$, 
\begin{equation} \label{eq:cluster-point-kac-rice}
    \begin{split}
        \bE[\eta(\eta-1)]= g_u^2\int_{[-g_u/2,g_u/2]^2} \bE \left[|\det \hess f(0)||\det \hess f(x)| \right . \\
        \left . \bI[f(0) \geq u, f(x)\geq u] \big | \nabla f(0)=0=\nabla f(x) \right] \psi_{x}(0,0)\td x \td y
    \end{split}
\end{equation}
where $p_{x}$ is the pdf of $(\nabla f(0), \nabla f(x))$. Now, the integrand of the above equation might blow up when $x$ is near $0$ since the vector $(\nabla f(0), \nabla f(x))$ becomes a degenerate Gaussian variable. 
To handle this, there is a trick called the divided difference method, introduced in \cite{cuzick_conditions_1975}. 

\pgap 

Let $(\bfe_1,\bfe_2)$ be an orthonormal basis of $\bR^2$ such that $x=|x| \bfe_1$. Then, as $x \to 0$, 
\begin{equation*}
    \begin{split}
        \det \cov (\nabla f(x), \nabla f(0))= & \|x\|^4 \det \cov(\nabla f(0), \dfrac{\nabla f(x)-\nabla f(0)}{\|x\|}) \\
                                        & \simeq \|x\|^4 \det \cov (\nabla f(0), \partial_{\bfe_1} \nabla f(0)).
    \end{split}
\end{equation*}

Now, conditionally on $\nabla f(x)=\nabla f(0)=0$, by the Taylor expansion, 
\begin{equation}
    \begin{split}
        \det \hess f(x) = & \det (\partial_{\bfe_1} \nabla f(x) \quad \partial_{\bfe_2} \nabla f(x)) \\
                        = & \|x\| \det \left ( \dfrac{\partial_{\bfe_1} \nabla f(x)- \partial_{\bfe_1} \nabla f(0)}{\|x\|} \quad \partial_{\bfe_2} \nabla f(x)  \right ) \\
                        \simeq    &\dfrac{ \|x\|}{2} \det (\partial_{\bfe_1}^2 \nabla f(x) \quad \partial_{\bfe_2} \nabla f(x)). 
    \end{split}
\end{equation}
Similarly, 
\begin{equation}
    \det \hess f(0) \simeq \dfrac{ \|x\|}{2} \det (\partial_{\bfe_1}^2 \nabla f(0) \quad \partial_{\bfe_2} \nabla f(0))
\end{equation}
Since 
\[
\psi_x(0,0)=1/\sqrt{\det( 2 \pi \cov(\nabla f(x), \nabla f(0)))}
\]
we have that, when $x$ is close to $0$,
\begin{equation}
    \begin{split}
         \bE \left[|\det \hess f(0)||\det \hess f(x)| \bI[f(0) \geq u, f(x)\geq u] \big | \nabla f(0)=0=\nabla f(x) \right] \psi_{x}(0,0) \\ 
         \simeq \dfrac{1}{4} \dfrac{\bE \left [ \det( \partial_{\bfe_1}^2 \nabla f(0), \partial_{\bfe_2} \nabla f(0))^2  \bI[f(0) \geq u] \big | \nabla f(0)=0= \partial_{\bfe_1} \nabla f(0) \right]}{\sqrt{ \det 2 \pi \cov (\nabla f(0), \partial_{\bfe_1} \nabla f(0))}}
    \end{split}
\end{equation}
hence it is bounded. 

\pgap 

This controls the integrand in \eqref{eq:cluster-point-kac-rice} close to the origin. Next, we estimate it away from the origin. For $\delta>0$, the pdf $\psi_x(0,0)$ is bounded above uniformly for $x \in B(0, \delta)^C$.

\pgap

We have, for $p,q>1$ with $1/p+1/q=1$, by H\"older's inequality, 
\begin{equation*}
    \begin{split}
        \bE [|\det \hess f(0)||& \det \hess f(x)| \bI[f(0) \geq u, f(x)\geq u] \big | \nabla f(0)=0=\nabla f(x) ] \leq  \\
        & (\bE[(|\det \hess f(0)||\det \hess f(x)|)^q \big |\nabla f(0)=0=\nabla f(x)])^{1/q} \\
       & (\bE[\bI[f(0) \geq u, f(x)\geq u]\big | \nabla f(0)=0=\nabla f(x) ])^{1/p}.
    \end{split}
\end{equation*}

From the divided difference method, observe that 
\begin{equation}
    \begin{split}
        (\bE[(|\det \hess f(0)||\det \hess f(x)|)^q \big |\nabla f(0)=0=\nabla f(x)])^{1/q} \psi_x(0,0)
    \end{split}
\end{equation}

is \textit{uniformly} bounded in $x \in \bR^2$ and for any fixed $q>1$. Hence we will bound this quantity in the integrand by a constant and focus on the quantity
\[
(\bE[\bI[f(0) \geq u, f(x)\geq u]\big | \nabla f(0)=0=\nabla f(x) ])^{1/p}.
\]

\pgap

Let us compute the following conditional Gaussian
\[
(f(0), f(x)) | \nabla f(0)=0, \nabla f(x)=0.
\]

Since we explicitly know the covariance structure of Bargmann-Fock field, the covariance matrix is,
\begin{equation}
    \begin{split}
        \begin{pmatrix} 1 & r(x)\\ r(x) & 1 \end{pmatrix} - 
        (1-e^{-t^2}(1+(t^2-1)^2)+e^{-2t^2(t^2-1)^2})^{-1} \times \\ \begin{pmatrix}
            t^2e^{-t^2}(1-e^{-t^2}) & t^2(t^2-1)(1-e^{t^2})e^{-5t^2/2} \\
            t^2(t^2-1)(1-e^{t^2})e^{-5t^2/2} & t^2e^{-t^2}(1-e^{-t^2})
        \end{pmatrix}
    \end{split}
\end{equation}
where $\|x\|^2=t^2$.

Observe that, as $t \to 0$,
\begin{equation}
    \begin{split}
        1-e^{-t^2}(1+(t^2-1)^2)+e^{-2t^2(t^2-1)^2} & = 3 t^4 + O(t^6) \\
        t^2e^{-t^2}(1-e^{-t^2}) & = t^4+O(t^6)
    \end{split}
\end{equation}

Hence, the variance of $f(0)$ conditional on $(\nabla f(0), \nabla f(x))$ near $x= 0$ is away from $1$. Since there are no degeneracies except when $x=0$, the covariance matrix of the conditional distribution is continuous with respect to $x$. 

So there's a constant $K>0$ such that for $x \in B(0,K)$ the conditional variance of $f(0)$ stays less than, say, $3/4$. So, 
\begin{equation} \label{eq:divided-diff-1}
    \begin{split}
        \bE[\bI[f(0)> u, f(x)>u | \nabla f(0)=0=\nabla f(x)]]&  \leq \bE[\bI[f(0)> u | \nabla f(0)=0=\nabla f(x)]] \\
        & \leq cu^{-1} \exp(-2u^2/3)
    \end{split}
\end{equation}
for $x \in B(0,K)$ and large enough $u$. 

\pgap 

Let $w(x)$ denote the correlation of $f(0)$ and $f(x)$ conditioned on $(\nabla f(0), \nabla f(x))$. We claim that there exists a constant $0<K_1<1$ such that $|w(x)| \leq K_1$ for all $x \in B(0,K)^C$. Since $f$ is a Gaussian field with spectral measure has a  a non-trivial open set in the support, we have that, for any distinct points $(x_1,x_2,\ldots, x_m)$
\[
(f(x_1),f(x_2), \ldots, f(x_m), \nabla f(x_1), \nabla f(x_2), \ldots, \nabla f(x_m))
\]
is non-degenerate. 

\pgap 

Hence, $|w(x)|\neq 1$ for $x \neq 0$. Also notice that variance of $(f(0), f(x))| \nabla f(0), \nabla f(x)$ tends to $1$ as $x \to \infty$ (because $t^n e^{-t^2/2} \to 0$ as $ t \to \infty$ for any $n$ integer). Similarly, the covariance tends to zero as $x \to \infty$. Hence $w(x) \to 0$ as $x \to \infty$. If a continuous function $b(x)$ on a compact set has $b<1$, then there's a constant $C<1$ such that $b(x)<C$ on that set. Hence, we proved the claim that there exists a constant $0<K_1<1$ such that $|w(x)| \leq K_1$ for all $x \in B(0,K)^C$. 

\pgap 

Using the bound \eqref{eq:divided-diff-1} in \eqref{eq:cluster-point-kac-rice} we have
\begin{equation}
    \begin{split}
        \bE[\eta(\eta-1)] \leq C g_u^2 \left ( u^{-1} \exp \left (-\dfrac{2u^2}{3p} \right )+ g_u^2 \exp \left (-\dfrac{u^2}{p(1+K_1)} \right ) \right ).
    \end{split}
\end{equation}

The first term on RHS of the bound is from integrating the integrand of the RHS of \eqref{eq:cluster-point-kac-rice} over $B(0,K)$ where the variance of $f(0)$ conditionally is around 2/3. The second term is naively bounding the integral of \eqref{eq:cluster-point-kac-rice} by the maximum times the volume. 

\pgap 

Take $p>1$ close enough so that, for some $0<\delta<2$
\[
\dfrac{1}{p(1+K_1)} \geq \dfrac{1}{2-\delta}.
\]

Choose $\tau(u)$ such that
\[
g_u=u \quad \text{ i.e. } \tau(u)=u^{3/2} e^{-u^2/4}.
\]

Finally, substituting all these for a large enough $R$, we get 

\begin{equation} \label{eqn-bound-cluster-points}
    \dfrac{1}{V_{R,u}} \bE \int_{D_{R,u}} (\Psi_{R,u}(A_{x})-1)\Psi_{R,u}( \td x) \leq C u \exp \left ( - u^2\dfrac{\delta}{2(2-\delta)}\right)
\end{equation}

The next task is to come up with a close coupling of $\Psi_{R,u}$ and $\Psi_{R,u}^0$ for large $R$.

\subsubsection{Coupling the field and its Palm version} \label{sec: coupling-palm}

Let $\tf$ be the Palm condition of the field $f$ to have a local maxima at $x=0$ with value at least $u(R)$. One way to couple the point processes $\Psi_{R,u}$ and $\Psi_{R,u}^0$ is to couple the fields $f$ and $\tf$, then consider the critical point processes corresponding to these fields. This section might be of general interest, so we consider stationary smooth Gaussian fields on $\bR^d$ rather than the particular Bargmann-Fock field.

\pgap 

Let us study the distribution of the Gaussian field $f$ conditioned on $x=0$ being a local maximum. Here, the field is conditioned in a way that it corresponds to the Palm conditioning (see the appendix). We follow the arguments in \cite[Section 1.2]{lindgren_local_1972}. 
The difference is that we require $f(0)\geq u$ rather than conditioning on $f(0)=u$. We also give a representation for the conditioned field for the case of the RPW model where $(f(x), \nabla f(x), \hess (f(x)))$ is degenerate. 

\pgap 

Consider the following event: 
\[
A(u,h) = \{f \text{ has a local maximum with height at least } u \text{ at some point } \mathbf{s} \text{ with } |\mathbf{s}| \leq h\}.
\]

Let $\chi=(x^1,x^2, \ldots, x^m) \in \bR^{d \times m}$ be $m$ distinct points in $\bR^d$ and $\tau=(t_1,t_2, \ldots, t_m) \in \bR^m$.
The goal is to compute the m-point marginal distribution 
\[
f(x^i), i=1,2\ldots m | A(u,h) \quad \text{ as } h \to 0.
\]
From \cite[Section 6.7, page 150]{adler_geometry_2010} it follows that 
\begin{equation} \label{conditional-convergence}
\bP(f(x^i)\leq t_i, i=1,2,\ldots m | A(u,h)) \to \dfrac{\bE[N_1(\tau, u)]}{\bE[N_1(u)]} \quad \text{as } h \to 0
\end{equation}
where 
\[
N_T(u) = \text{ the number of local maxima } \mathbf{s} \text{ with } |\mathbf{s}| \leq T \text{ and } f(\mathbf{s}) \geq u, \text{ and }
\]
\[
N_T(\tau,u)= \text{ those maxima } \mathbf{s} \text{ for which } f(\mathbf{s}+x^i) \leq t_i, i=1,2,\ldots m.
\]
Here's a brief reason why. Define 
\[
A(\tau ,u, h) = \bigl\{ f \text{ has a local maximum in } B(0,h) \text{ of height at least } u,
\text{ and } f(x^i) \le t_i,~ i=1,\dots,k \bigr\}.
\]

\noindent
By stationarity and localization, for small \( h \),
\[
\frac{\mathbb{P}(A(\tau ,u, h))}{h^d}
\approx \mathbb{E}\, N_1(\tau, u),
\qquad
\frac{\mathbb{P}(A(u,h))}{h^d}
\approx \mathbb{E}\, N_1(u).
\]
Hence
\[
\frac{\mathbb{P}(A(\tau,u, h))}{\mathbb{P}(A(u,h))}
~\longrightarrow~
\dfrac{\bE[N_1(\tau, u)]}{\bE[N_1(u)]}.
\]

Under ergodicity, the same ratio also represents the almost-sure limit of these configurations around high peaks. To compute explicitly the RHS of \eqref{conditional-convergence}, we need some regularity properties of the field $f$.

\paragraph{The non-degenerate case:} Assume that $(f(x), \nabla f(x), \hess (f(x))$ is a non-degenerate Gaussian vector. Here, the Hessian is vectorised as follows,
\[
\hess(f(x))= (\partial_{ii}f(x), i=1,2,\ldots d, \partial_{ij}f(x), 1 \leq i<j \leq d).
\]
Denote the covariance matrix of the vector $(f(0),\nabla f(0), \hess (f(0)))$, partitioned in a natural way, by
\begin{equation}
   \Sigma= \begin{pmatrix}
        1 & 0 & S_{02} \\
        0 & S_{11} & 0 \\
        S_{20} & 0 & S_{22}
    \end{pmatrix}
\end{equation}
where $S_{11}, S_{22}$ are the internal covariance matrices of $\nabla f(0)$ and $\hess (f(0))$ respectively. Define 
\[
S_1(x)=(-r_i(x), 1\leq i \leq d),
\]
\[
S_2(x)= (r_{ii}(x), 1\leq i \leq d, r_{ij}(x), 1 \leq i < j \leq d)
\]
where $r_i= \partial_i r$ and $r_{ij}=\partial_{ij}r$. The plan is to use the Kac-Rice formula to get an expression for $\bE[N_1(\tau,u)], \bE[N_1(u)]$, for which we need more notation. 

\pgap 

Define the following probability densities and (Gaussian) conditional densities
\begin{equation*}
\begin{matrix}
    p_{\chi}(t,v,Z,x) & \text{ for } & f(0), \nabla f(0), \hess(f(0)), f(x^1), f(x^2), \ldots, f(x^m) \\
    p_{\chi}(x| t,v,Z) & \text{ for } & f(x^1), f(x^2), \ldots, f(x^m) | f(0)=t, \nabla f(0)=v, \hess(f(0))=Z \\
    p(t,v,Z) & \text{ for } & f(0), \nabla f(0), \hess (f(0)) \\
    p(Z,t|v) & \text{ for } & \hess(f(0)), f(0) | \nabla f(0)=0.
\end{matrix}
\end{equation*}

Then by the Kac-Rice formula, we have 
\begin{equation}
    \dfrac{\bE[N_1(\tau, u)]}{\bE[N_1(u)]}= \dfrac{\int_{y_1 \leq t_1} \cdots \int_{y_m \leq t_m} \int_{t\geq u} \int_{Z \prec 0} |\det Z| p_{\chi}(t,0,Z,y) dZdtdy}{\int_{t \geq u} \int_{Z \prec 0}|\det Z|p(t,0,Z) dZdt}
\end{equation}

where $Z \prec 0$ means $Z$ is a negative definite matrix. Hence the conditional distribution of $f(x^1),\ldots, f(x^m)$ given that $f$ has a local maxima with height at least $u$ at $0$ has the density
\begin{equation} \label{conditional-density}
    \dfrac{\int_{t \geq u} \int_{Z \prec 0}|\det Z| p_{\chi}(t,0,Z,y) dZdt}{\int_{t \geq u} \int_{Z \prec 0}|\det Z|p(t,0,Z) dZdt}.
\end{equation}

Define 
\begin{equation} \label{eq:def-q_u}
    q_u(t,Z)= \dfrac{\det Z p(Z,t|0)}{\int_{t \geq u}\int_{Z \prec 0}\det Z p(Z,t|0) dZdt} \bI[Z \prec 0, t \geq u].
\end{equation}

Hence, the density \eqref{conditional-density} is 
\[
\int_{t \geq u} \int_{Z \prec 0} q_u(t,Z) p_{\chi}(\tau | t,0,Z) dZdt.
\]

So, `freezing' the values of $f(0), \nabla f(0), \hess (f(0))$ the joint distribution of $f(x^1),\ldots, f(x^m)$ is Gaussian. Let us compute the mean and covariance matrix of the Gaussian with density $p_{\chi}(\cdot | t, 0, Z)$. 
Define the following quantities $A(x) \in \bR , b(x) \in \bR^{d(d+1)/2}$:
\[
(A(x), b(x))= (r(x), S_2(x)) \begin{pmatrix}
    1 & S_{02} \\
    S_{20} & S_{22}
\end{pmatrix}^{-1}.
\]

By the Gaussian regression formula, we have that the mean of a Gaussian with density $p_{\chi}(\cdot | t, 0, Z)$ is 
\[t A(x^i)Z^Tb(x^i) \text{ for } i=1,2,\ldots m . \] 
To compute the covariance matrix, we use the fact that $(f(0), \hess(f(0)))$ is independent of $\nabla f(0)$ by stationarity of the field $f$.
Again, by the Gaussian regression formula, the $(i,j)$th entry of the covariance matrix is 
\begin{equation}
C(x^i,x^j)= r(x^i-x^j)-(r(x^i),S_2(x^i)) \begin{pmatrix}
    1 & S_{02} \\
    S_{20} & S_{22}
\end{pmatrix}^{-1} \begin{pmatrix}
    r(x^j) \\
    S_{2}(x^j)
\end{pmatrix} - S_1(x^i) S_{11}^{-1} S_1(x^j).
\end{equation}

This proves the following proposition. 

\begin{proposition}(c.f. \cite[Theorem 1.2]{lindgren_local_1972}) \label{prop:conditional-dist-field}
    Given a local maximum at $x=0$ with height at least $u$, the conditional process $f(x)$ has the same law as the following process: 
    \[
    \tf(x)= \xi A(x) + Z b(x) + \barf(x)
    \]
    where $\barf$ is a non-stationary zero mean Gaussian field with covariance function $C(\cdot, \cdot)$ and $(\xi,Z)$ is random vector with density $q_u$ which is independent of $\barf$.
\end{proposition}

Let us couple the fields $\tf$ and $f$ such that $\|\tf(x)-f(x)\|$ is small for $\|x\| \gg 1$. Recall the construction of the reproducing kernel Hilbert space (RKHS) of the kernel $r$. 
Let $\rho$ be the spectral measure associated with the kernel $r$, then the RKHS $H$ is given by 
\[
H= \mathcal{F}L^2_{sym}(\rho)
\]
where $\mathcal{F}$ denotes the Fourier transform, and $L^2_{sym}(\rho)$ are the symmetric $L^2(\rho)$ functions. The inner product of the Hilbert space $H$ is given by, 
\[
\langle \mathcal{F}\phi, \mathcal{F}\psi \rangle_{H} = \langle \phi, \psi \rangle_{L^2(\rho)}.
\]
Now the field $f$ is a white noise in the RKHS $H$, i.e. for an orthonormal basis $\{\psi_i\}$ of $H$, formally we have
\[
f=\sum_{i} a_i \psi_i
\]
where $a_i$'s are i.i.d standard Gaussian variables. 

\pgap 

Since we assume that the field $f$ is $C^3$-smooth, we have up to 6 finite moments of $\rho$. This immediately implies that all of $r, r_i (1 \leq i \leq d), r_{ij} (1 \leq i, j \leq d)$ belong to $H$ since 
\[
r_i= \int x_i e^{ \sqrt{-1} x \cdot} d \rho(x) , \quad r_{ij}= \int x_i x_j e^{\sqrt{-1} x \cdot} d \rho(x).
\]

We have the following observation using the Fourier representation for $r$ and its derivatives, 
\begin{align*}
    \langle r, r_i \rangle_{H}= -r_i(0)=0 \quad & \text{ for } 1\leq i \leq d \\
    \langle r_{kl}, r_i \rangle_{H}=-r_{ikl}(0)=0 \quad & \text{ for } 1 \leq i,k,l \leq d \\
    \langle r, r_{ij} \rangle_{H}= r_{ij}(0) \quad & \text{ for } 1 \leq i,j \leq d \\
    \langle r_{ij}, r_{kl}\rangle_{H}= r_{ijkl}(0) \quad & \text{ for } 1 \leq i,j,k,l \leq d.
\end{align*}

Observe that the covariance matrix for $(f(0),\nabla f(0), \hess(f(0)))$ which is $\Sigma$ is given by the same entries, i.e. 
\[
\bE[f(0) \partial_{x_i}f(0)]=-r_i(0)=0, \quad  \bE[f(0) \partial_{x_ix_j}f(0)]=r_{ij}(0)
\]
and so on. Hence $\{r, r_i, r_{ij}: 1 \leq i,j \leq d\}$ is a linearly independent set in $H$. Otherwise, it would imply that the covariance matrix $\Sigma$ is degenerate. 

\pgap 

Now by Gram-Schmidt orthonormalisation of the set $\{r, r_i, r_{ij}: 1 \leq i,j \leq d\}$ with respect to the inner product $\langle,\rangle_{H}$, say you get the orthonormal set 
\[ \{\phi_0, \phi_1, \cdots, \phi_d, \phi_{11}, \cdots, \phi_{ij}, \cdots, \phi_{dd}\} \]
which is presented by 
\[
\begin{pmatrix}
    \phi_0 \\ \phi_1 \\ \vdots \\ \phi_d \\ \phi_{11} \\ \vdots \\ \phi_{ij} \\ \vdots \\ \phi_{dd}
\end{pmatrix} = \mathbf{B} \begin{pmatrix}
    r \\ r_1 \\ \vdots \\ r_d \\ r_{11} \\ \vdots \\ r_{ij} \\ \vdots \\ r_{dd}
\end{pmatrix}
\]

Let $b_1, b_2, \ldots, b_n$ be the rows of the matrix $\mathbf{B}^{-1}$ (so $n =1+d+d(d+1)/2$). Then by the construction of the $\phi_i$'s, we have that 
\[
b_i \cdot b_j = \langle r_{\#}, r_*\rangle_{H}
\]
where the inner product on LHS is the Euclidean one in $\bR^n$ and indices $\#, *$ are the corresponding ones $i,j$ (say $i=1$ corresponds to $r_{\#}=r$, $i=1+d+d(d+1)/2$ corresponds to $r_{\#}=r_{dd}$). As observed before, the entries of $\Sigma$ are given by the inner products $\langle r_{\#}, r_*\rangle_{H}$, hence we have
\[
\mathbf{B}^{-1}(\mathbf{B}^{-1})^T=\Sigma. 
\]

\pgap 

Extend the orthonormal set $\{\phi_0, \cdots, \phi_{dd}\}$ to an orthonormal basis $\{\psi_i\}_{i=0}^{\infty}$ of $H$, where the first $n=1+d+d(d+1)/2$ elements match. Define the field $f$ as 
\begin{equation} \label{eqn:field-def}
f= \sum_{i=0}^{\infty} a_i \psi_i
\end{equation}
where $a_i$'s are i.i.d standard Gaussian variables. Define 
\begin{equation} \label{eqn:residue-field-defn}
\barf = \sum_{i \geq n} a_i \psi_i 
\end{equation}
so that 
\[
f = \begin{pmatrix}
    a_0 & a_1& \cdots & a_{n-1}
\end{pmatrix} \begin{pmatrix}
    \psi_0 \\ \psi_1 \\ \vdots \\ \psi_{n-1}
\end{pmatrix} + \barf. 
\]

Now we have, 
\begin{align*}
\cov(f(x)-\barf(x), f(y)-\barf(y))= & \begin{pmatrix}
    \psi_0(x) & \psi_1(x) & \cdots & \psi_{n-1}(x)
\end{pmatrix} \begin{pmatrix}
    \psi_0(y) \\ \psi_1(y) \\ \vdots \\ \psi_{n-1}(y)
\end{pmatrix} \\
= & \begin{pmatrix}
    r(x) & \cdots & r_{dd}(x)
\end{pmatrix} \mathbf{B}^T\mathbf{B} \begin{pmatrix}
    r(y) \\ \vdots \\ r_{dd}(y)
\end{pmatrix} \\
=& \begin{pmatrix}
    r(x) & \cdots & r_{dd}(x)
\end{pmatrix} \Sigma^{-1}\begin{pmatrix}
    r(y) \\ \vdots \\ r_{dd}(y)
\end{pmatrix} .
\end{align*}

Hence, $\barf$ defined here has the covariance $C(,)$ as in Proposition \ref{prop:conditional-dist-field}. 

\pgap 

Now, let $f$ be the field defined as in \eqref{eqn:field-def} and $\barf$ be as in \eqref{eqn:residue-field-defn}, i.e. using the orthonormal basis $\{\psi_i\}_{i=0}^{\infty}$ and the noise $(a_0,a_1,a_2, \ldots)$. Let $(\xi,Z)$ be a random vector as in Proposition \ref{prop:conditional-dist-field} which is independent of the sequence $a_0,a_1,a_2, \ldots$. 
Then, 
\[
\tf= \xi A(x) - Z b(x) + \barf(x)
\]
is a random function which is $C^2$-smooth, non-Gaussian, has a  local maxima at $x=0$ with height at least $u$.

\paragraph{Isotropic case:} Consider a Gaussian function $f:\bR^d \to \bR$ which is $C^2$-smooth, stationary, unit variance. Additionally assume that $(f(x),\nabla f(x))$ is non-degenerate for $x \in \bR^d$. 
By stationarity, we can prove that $(f(x), \hess(f(x)))$ is independent of $(\nabla f(x))$. Hence, if the vector $(f(x), \nabla f(x), \hess(f(x)))$ is degenerate then it implies that $(f(x), \hess (f(x)))$ is degenerate. 

\pgap 

Assume that $(f(0), \hess(f(0)))$ is degenerate. Under mild symmetry conditions on the covariance kernel, such as for $d=2, r(x_1,x_2)=r(x_1,-x_2)$,  we can show that $f$ is a monochromatic random wave up to linear change of coordinates of the domain (i.e. the spectral measure of $f$ is supported on $\mathbb(S)^{d-1} \subset \bR^d$). 
If we insist that $f$ is isotropic, then $f$ has spectral measure which is uniform on $\mathbb{S}^{d-1}$ ($d=2$ case corresponds to RPW model) [see \cite[Proposition 3.3]{cheng_expected_2018}].

\pgap 

In this section assume that $f$ is a Gaussian field with spectral measure given by the  uniform measure on $\mathbb{S}^{d-1}$. The construction of the coupling of the conditioned field and the original field is similar to that of the non-degenerate case above, so we will not give all the details. 
Also, we retain the same notation as in the previous section. 

\pgap 

We will work with the vector $(\nabla f(x), \hess(f(x)))$ which is non-degenerate, instead of \[(f(x), \nabla f(x), \hess(f(x))).\] Recall that $\chi=(x^1,x^2, \ldots, x^m) \in \bR^{d \times m}$ are $m$ distinct points in $\bR^d$ and $\tau=(t_1,t_2, \ldots, t_m) \in \bR^m$.
Rename/redefine the following quantities while keeping the other notation the same. Let 
\[
\Sigma= \begin{pmatrix}
    S_{11} & 0 \\
    0 & S_{22}
\end{pmatrix}
\]
where $S_{11}$ and $S_{22}$ are covariance matrices of $\nabla f(x)$ and $\hess(f(x))$ respectively. 

Define the following probability densities and (Gaussian) conditional densities
\begin{equation*}
\begin{matrix}
    p_{\chi}(v,Z,x) & \text{ for } &  \nabla f(0), \hess(f(0)), f(x^1), f(x^2), \ldots, f(x^m) \\
    p_{\chi}(x|v,Z) & \text{ for } & f(x^1), f(x^2), \ldots, f(x^m) |\nabla f(0)=v, \hess(f(0))=Z \\
    p(v,Z) & \text{ for } & \nabla f(0), \hess (f(0)) \\
    p(Z|v) & \text{ for } & \hess(f(0)) | \nabla f(0)=v.
\end{matrix}
\end{equation*}

\pgap

Noting that $f(x)=-\tr(\hess(f(x)))$, we have, again by the Kac-Rice formula, 
\begin{align*}
    \dfrac{\bE[N_1(\tau, u)]}{\bE [N_1(u)]}=& \dfrac {\bE \left [|\det \hess(f(0))| \bI[Q] \bigg | \nabla f(0)=0 \right ]}{\bE[|\det \hess(f(0))| \bI[\text{ind}(\hess (f(0)))=d] \bI[f(0) \geq u] ]} \\
    = & \dfrac{\int_{y_1 \leq t_1} \cdots \int_{y_m \leq t_m}  \int_{Z \prec 0, \tr(Z) \leq -u} |\det Z| p_{\chi}(0,Z,y) dZdy}{ \int_{Z \prec 0, \tr(Z)\leq -u}|\det Z|p(0,Z) dZ}
\end{align*}

where \[
\bI[Q] = \bI[\text{ind}(\hess(f(0)))=d] \bI[f(0) \geq u] \prod_{i=1}^m \bI[f(x^i) \leq t_i]. 
\]

Following a similar argument as in the non-degenerate case, we get the following result regarding the conditional distribution of $f$. 

\begin{proposition}
    Let $f$ be an isotropic Gaussian field with spectral measure being uniform measure on $\mathbb{S}^{d-1}$, and other assumptions on $f$ as in the beginning of this subsection. Then, given a local maximum with height at least $u$ at the origin, the conditional process $f(x)$ has the same law as the following process: 
    \[
    \tf(x)=  Z S_{22}^{-1} S_2(x)' + \barf(x)
    \]
    where $\barf$ is a non-stationary zero mean Gaussian field with covariance function $C(\cdot, \cdot)$ and $Z$ is random vector with density $q_u$ which is independent of $\barf$. Here 
    \[
    q_u(Z):= \dfrac{\det Z p(Z|0)}{\int_{Z \prec 0, \tr(Z) \leq -u} \det Z p(Z|0) dZ} \bI[Z \prec 0, \tr(Z) \leq -u] \quad \text{and}
    \]
    \[
    C(x,y)= r(x-y) - S_2(x) S_{22}^{-1}S_2(y)' - S_1(x)S_{11}^{-1} S_1(y)'.
    \]
\end{proposition}

Let us now couple the conditioned field $\tf$ and the original field $f$ such that $\|f(x)-\tf(x)\|$ is small for $\|x\| \gg 1$. Again, we follow the argument of the non-degenerate case. The only difference is that instead of extending the orthonormal basis generated by 
\[
\{r, r_i, r_{ij}: 1\leq i,j \leq d \}
\]
we look at the linearly independent set $\{r_i, r_{ij}: 1\leq i,j \leq d\}$ because $r=-(r_{11}+r_{22}+ \cdots+ r_{dd})$. We take $n=d+d(d+1)/2$ (one less than previously defined), and orthonormal basis $\{\psi_i\}_{i=1}^{\infty}$ of the RKHS $H$, the noise $(a_1,a_2, \ldots)$ to define 
\[
f = \sum_{i=1}^{\infty} a_i \psi_i \quad \text{ and } \quad \barf= \sum_{ i> n} a_i \psi_i.
\]
Define the conditioned field $\tf$ to be 
\[
\tf= Z S_{22}^{-1} S_2(x)' + \barf(x)
\]
where $Z$ has the density $q_u$ and is independent of the noise $(a_1,a_2, \ldots)$. Verification that $\barf$ has the desired covariance structure is the same as in the non-degenerate case.

\subsubsection{Difference in number of points in perturbed fields}

Now that we have a coupling of the field $f$ and its Palm version $\tf$ at $x=0$, define $\Psi_{R,u}^0$ as the rescaled point process of local maxima of $\tf$ above $u(R)$ in the domain $D_{R,u}$, just as $f$. This is a coupling of  $(\Psi_{R,u}^0, \Psi_{R,u})$.

\pgap 

Define the following functions:
\[
F_t(x):= f(x)+th(x), \text{ and } h(x)= \tf(x)-f(x), \quad t \in [0,1].
\]
We will analyse the quantity $\rho_1(\Psi_{R,u}^0, \Psi_{R,u})$ by comparing the critical point structure of $F_0$ and $F_1$ in the box $[-R/2,R/2]^2$.
One way to compare the critical point structures of $F_0$ and $F_1$ is to interpolate (as defined above) and see how the positions of the critical points change. 

\pgap 

Define \textit{a continuous flow $x_t$ of a critical point under $F_t$} to be a continuous function $x.: [0,\delta] \to \bR^2$ with $\nabla F_t(x_t)=0$, for some $0< \delta \leq 1$.

\pgap 

\textbf{Almost sure properties of $F_t$}: 

\begin{enumerate}
    \item $F.$ is a $C^2$-smooth function in $(t,x)$ almost surely since both $f, \tf$ are at least $C^2$-smooth in $x$.
    \item Let us define the flow, for $x_t \in \bR^2$,
    \begin{equation} \label{eq:critical-flow}
        \dfrac{dx_t}{dt}=- \hess(F_t(x_t))^{-1} \cdot \nabla h(x_t), \quad t \in [0,1].
    \end{equation}
    This defines a smooth flow away from non-invertible points of $\hess(F_t)$. By applying chain rule, you can check that $\nabla F_t(x_t)$ is constant in $t$. Hence, if you start the flow with a critical point of $F_0$, then $x_t$ is a critical point of $F_t, \forall t$ as long as the flow is defined. 
    \item The index of a critical point $x_0$ changes during the $F_t$-flow only if 
    \[ \det(\hess(F_t(x_t)))=0\text{ for some } t \in (0,1). \] 
    This is because to change the index, one of the eigenvalues of $\hess(F_t)$ has to change the sign, hence the determinant has to be zero at some $t \in (0,1)$.
    \item Note that a critical point of $F_t$ is an intersection of the hypersurfaces $\{\partial_{x_i}F_t=0\}, i=1,2$. This intersection is stable if it is a transversal intersection (in this case, transversal intersection corresponds to non-degeneracy of the Hessian of $F_t$ at that point). Hence, there's creation/destruction of a critical point under perturbation only if 
    \[\det(\hess(F_t))=0. \]
    \item Also, two critical points of $F_t$ cannot merge into one during the flow unless $\hess(F_t)$ is degenerate at some point in between. This is because if the intersection of the hypersurfaces $\{\partial_{x_i}F_t=0\}, i=1,2$ is transversal, then the number of intersection points (which corresponds to critical points in this case) is stable under perturbation ( as it is a topological quantity). See Chapter 2 of \cite{guillemin_differential_1974} for details. 
\end{enumerate}

Now, the quantity $\bE||\Psi_{R,u}^0(A_0^c)|-|\Psi_{R,u}(A_0^c)||$ is bounded by the sum of expectations of number of points of the following types (let $D'_R= [-R/2,R/2]^2 \setminus (\mu(u)A_0)$): 
\begin{enumerate}
    \item $\{x \in D'_R : \hess(F_t(x)) \text{ is degenerate for some } t \in [0,1], f(x)> u, \tf(x)>u \}$ 
    
    This set corresponds to the degeneracies that arise during the critical point flow, including the creation/destruction of points as described above.
    \item $\{x \in D'_R : \nabla F_t(x)=0 \text{ for some } t \in [0,1], f(x)=u \text{ or } \tf(x)=u\}$

    These are the points which flow in/out of the domain $\{f(x)>u\} \cap \{\tf(x)>u\}$

    \item $\{x \in \partial(D'_R): \nabla F_t(x)=0 \text{ for some } t \in [0,1], f(x) \geq u \text{ or } \tf(x) \geq u \}$

    Points which flow in/out of the domain $[-R/2,R/2]^2 \setminus (\mu(u) A_0)$. 
    \item $\{x \in D'_R: \nabla f(x)=0 \text{ or } \nabla \tf(x)=0, x \in (\{f>u\}\cap \{\tf<u\})\cup (\{f<u\} \cap \{\tf >u\}) \}$.
\end{enumerate}

We will use the Kac-Rice formulas to bound the number of the above-listed points, named $P_i$ for $i=1,2,3,4,5$ as below.  Recall the notation from the coupling section above and define 
\[
f_1(x):= f(x)-\barf (x), \quad f_2(x)= \tf(x)- \barf (x), \quad x \in \bR^d. 
\]
Note that the random functions $f_1(x), f_2(x), \barf (x)$ are jointly independent. Consequently, all derivatives of these functions are jointly independent. We need the following lemma to ensure the existence of some of the densities of the random vectors (i.e. non-degeneracy conditions) to be used in Kac-Rice formula. 

\begin{lemma} \label{lemma-non-degeneracy}
    For large enough $R, u$, for all $x \in D'_R$ the vector 
    \[(\barf (x), \nabla \barf(x), \hess \barf (x))\] is a non-degenerate Gaussian vector in $\bR^6$.
\end{lemma}

\begin{proof}
    Notice that \[ (f(x), \nabla f(x), \hess f(x))= ( f_1(x), \nabla f_1(x), \hess f_1(x)) + (\barf (x), \nabla \barf(x), \hess \barf (x)). \] Let
    \begin{align*}
        X & = ( f_1(x), \nabla f_1(x), \hess f_1(x)) \\
        Y & = (\barf (x), \nabla \barf(x), \hess \barf (x)) \\
        Z & = (f(x), \nabla f(x), \hess f(x)) 
    \end{align*}
    and $\Sigma_X, \Sigma_Y, \Sigma_Z$ be their respective covariance matrices. We want to show that $\Sigma_Y$ is invertible matrix for large $\|x\|$. By construction, $f_1(x)$ is a linear combination of $r(x)$ and its derivatives up to second order with Gaussian coefficients. As a result, each entry of the covariance matrix $\Sigma_X$ of $( f_1(x), \nabla f_1(x), \hess f_1(x))$ is bounded above by $c \|r(x)\|^2_4$, where 
    \[\|r(x) \|_4:= \max_{0 \leq |\alpha| \leq 4} |\partial^{\alpha} r(x)|.\]
    Hence, the maximum value of absolute eigenvalue of $\Sigma_X$ is at most $c \|r(x)\|^2_4$. Since $r(x)=e^{-\|x\|^2/2}$, we have 
    \[\|r(x)\|_4 \to 0 \quad \text{ as } \|x\| \to \infty.\]
    Since $f(x)$ is stationary field the covariance matrix $\Sigma_Z$ is constant in $x$. One of the property of Bargmann-Fock field is that $Z$ is invertible. Let $c_0$ be the lowest absolute eigenvalue of $\Sigma_Z$.  Now for any unit vector $\theta  \in \bR^6$, since $f_1$ and $\barf$ are independent functions
    \begin{align*}
        \theta^T \Sigma_Z \theta = \theta^T \Sigma_X \theta + \theta^T \Sigma_Y \theta.  
    \end{align*}
    For any $\ep >0$, for large enough $\|x\|$, we have $\theta^T \Sigma_X \theta < \ep$ for all $\theta$. Also, $\theta^T \Sigma_Z \theta > c_0$ for all unit vector $\theta$. Hence, $\theta^T \Sigma_X \theta \geq c_0 -\ep$ which proves our claim. 
\end{proof}

Notice that for any independent random vectors $X_1$ and $Y_1$, if $X_1$ has density in $\bR^d$ then $X_1+Y_1$ also has density in $\bR^d$. Indeed, if $\rho_{X_1}$ is the density of $X_1$ and $\mu_{Y_1}$ is the law of $Y_1$ then the convolution $\rho_{X_1}* \mu_{Y_1}$ is the density of $X_1+Y_1$. Combining this observation and Lemma \ref{lemma-non-degeneracy} we can show that, for instance, $(f(x), \nabla F_t(x),\hess F_t(x))$ has density for all $t$ and $ x \in D'_R$  because 
\begin{multline*}
    (f(x), \nabla F_t(x),\hess F_t(x)) = (\hess \barf(x), \nabla \barf (x), \barf (x))+ \\ (f_1(x),(1-t)\nabla f_1(x)+t \nabla f_2(x), (1-t)\hess f_1(x)+t \hess f_2(x) ).
\end{multline*}  

\pgap

\pgap

\textbf{Bound on $P_1$:} Define 
\[
P_1:=\{(t,x) \in [0,1]\times D'_R: \nabla F_t(x)=0, f(x)=u\}
\]
and we want to compute $\bE[P_1]$. By the Kac-Rice formula, we have
\begin{equation} \label{eq: P_1-kac-rice}
    \bE[P_1]= \int_0^1 \int_{D'_R} \bE[|\det (\td (\nabla F_t(x),f(x)))|\big | \nabla F_t(x)=0, f(x)=u] p_{t,x}(0,u) \td x \td t
\end{equation}
where $p_{t,x}$ is the pdf of $(\nabla F_t(x), f(x))$. Since $(f_1, f_2, \barf)$ are independent, we first condition on $f_1, f_2$ and compute the integrand above. Let 
\[
(f_1(x),\nabla f_1(x), \hess f_1(x))=(a_1,v_1, M_1) \quad (f_2, \nabla f_2(x), \hess f_2(x))=(a_2,v_2,M_2).
\]
Then the integrand in \eqref{eq: P_1-kac-rice} is, 
\begin{equation} \label{eqn: cond-hess-1}
    \begin{split}
      Z(t, a_1, v_1, M_1, v_2, M_2)=  \bE \left [|t||\det \begin{pmatrix}
            (v_2-v_1) & 0 \\
            \hess \barf (x) + (1-t)M_1+tM_2 & (v_2-v_1)^T 
        \end{pmatrix}| \big | \right. \\  \left . \nabla \barf(x)=(t-1)v_1-tv_2, \barf(x)=u-a_1 \right ]
    \end{split}
\end{equation}
integrated over the vectors $(f_1(x),\nabla f_1(x), \hess f_1(x))$ and $(f_2, \nabla f_2(x), \hess f_2(x))$.

Now, we look at the distribution of 
\[
\hess \barf(x) | \barf(x)=\bara, \nabla \barf (x)=\barv. 
\]

We have that, since $f$ is Bargmann-Fock field in 2-dimensions, for $p_1=(x_1,y_1), p_2=(x_2,y_2)$
\begin{multline} \label{eq: barf-covariance}
   C(p_1,p_2):= \bE[\barf(p_1) \barf(p_2)]= \exp \left (-\dfrac{1}{2}((x_1-x_2)^2+(y_1-y_2)^2) \right ) - \\ \exp \left (-\dfrac{1}{2}(x_1^2+y_1^2+x_2^2+y_2^2)\right ) \left [1+\dfrac{1}{2}(x_1x_2)^2+\dfrac{1}{2}(y_1y_2)^2+2x_1y_1x_2y_2 \right ] 
\end{multline}
which follows from Section \ref{sec: coupling-palm}. Denote the covariance of partial derivatives by subscripts to $C$, for example, 
\[
C_{11,2}(p_1,p_2):=\bE[\partial_{11} \barf(p_1) \partial_{2} \barf(p_2)].
\]

When $p_1=p_2$, by abuse of notation, we write $C(p_1, p_1)=C(p_1)$. 

Let 
\[ \|r(x)\|_4 := \max_{0\leq |\alpha| \leq 4} |\partial^{\alpha} r(x)|. \]

Write $\hess \barf =( \partial_{11} \barf, \partial_{22} \barf , \partial_{12} \barf )$. We choose $A_0$ such that, on $(R/\sqrt{N}A_0)^C$, $\|r(x)\|_{C^4}$ is small. 
Hence we compare covariances $\cov(\hess \barf (x), (\barf(x), \nabla \barf(x)))$ with that of
$\cov(\hess f(x), (f(x), \nabla f(x)))$.

\pgap 

Observe that $(f(x), \hess f(x))$ and $\nabla f(x)$ are independent. Also, $\cov(\partial_{11}f(x), f(x)) = \cov(\partial_{22}f(x), f(x))=-1$ and $\cov(\partial_{12}f(x), f(x))=0$. 

\pgap 

We have, 
\[
\Sigma_{22}=\Var(\barf(x), \nabla \barf(x))) = I + O(\|r(x)\|_4) \bOne
\]
and 
\[
\Sigma_{12}=\cov(\hess \barf(x), (\barf(x), \nabla \barf(x))) = \begin{pmatrix}
    -1 & 0 & 0 \\
    -1 & 0 & 0 \\
    0 & 0 & 0
\end{pmatrix} + O(\|r(x)\|_4) \bOne
\]

where $\bOne$ are matrices with all entries $1$, with suitable dimensions.  So, 
\[
\hess \barf(x) | \barf(x)=\bara, \nabla \barf (x)= \barv 
\]

is a Gaussian vector with mean 
\[
\Sigma_{12} \Sigma_{22}^{-1} \begin{pmatrix}
    \bara \\ \barv 
\end{pmatrix} = \begin{pmatrix}
    - \bara \\ - \bara \\ 0
\end{pmatrix} + O(\|r(x)\|_4) \bOne \begin{pmatrix}
    \bara \\ \barv 
\end{pmatrix}. 
\]

By standard Gaussian regression, 

\[
(\hess \barf(x) | \barf(x)=\bara, \nabla \barf (x)= \barv ) - \Sigma_{12} \Sigma_{22}^{-1} \begin{pmatrix}
    \bara \\ \barv 
\end{pmatrix}
\]

is a Gaussian vector which does not depend on $\bara$ or $\barv$. And, the variance of this vector stays away from zero for $x \in D'_R$. 

\pgap 

Observe that the determinant of the form 
\[
\det \begin{pmatrix}
    X & 0 \\
    M & Y 
\end{pmatrix}
\]
is a bilinear form in $(X,Y)$ and since $M$ is $2 \times 2$ matrix, coefficients of this bilinear form will be linear in entries of $M$. 
Hence, the expression \eqref{eqn: cond-hess-1} is a quadratic form in $(v_2-v_1)$ with coefficients being degree one polynomials in $(u-a_1,(t-1)v_1-tv_2, M_1, M_2)$. 
So, we have integrated out $(\barf(x), \nabla \barf(x), \hess \barf (x))$.

\pgap 
Our next goal is to compute the following quantity, 
\begin{equation} \label{eq: gauss-regression-kac-rice-1}
\bE[Z(t, f_1(x), \nabla f_1(x), \hess f_1(x), \nabla f_2(x), \hess f_2(x)) p((t-1) \nabla f_1(x)-t \nabla f_2(x), u - f_1(x))]
\end{equation}
where $p$ is now the pdf of the Gaussian vector $(\nabla \barf (x), \barf(x))$. Again, since $f_1, f_2$ are independent random functions, first we will integrate over $f_1$ and its derivatives. 
Observe that, since $\nabla \barf(x)$ is a zero mean Gaussian vector  
\[
p((t-1) \nabla f_1(x)-t \nabla f_2(x), u - f_1(x)) \leq p(0, u - f_1(x)). 
\]
We know that conditioning on a Gaussian random variable only decreases its variance. Hence, 
\[
p(0, u - f_1(x)) \leq \overline{P}(u - f_1(x))
\]

for $u$ large enough, where $\overline{P}$ is the pdf of $\barf(x)$.

\pgap 

Now the Gaussian vector $((f_1(x),\nabla f_1(x), \hess f_1(x))$ is a linear transform of i.i.d standard Gaussians with coefficients $\partial^{\alpha} r(x),0\leq |\alpha| \leq 4 $. Conditioning $(\nabla f_1(x), \hess f_1(x))$ on $f_1(x)=a_1$ only shifts the variance matrix by order $O(\|r(x)\|_4)$ and the mean $a_1$.
After integrating out $(\nabla f_1(x), \hess f_1(x))$ in expression \eqref{eq: gauss-regression-kac-rice-1} while conditioning $f_1(x)=a_1$, we get (in terms of $a_1$)
\begin{equation} \label{eq: f_1-int}
    w(a_1) \exp \left (- \dfrac{1}{2 \sigma^2(x)}(u-a_1)^2 \right)
\end{equation}

where $w$ is a polynomial of degree at most 2 and $\sigma^2(x)= \Var (\barf(x))$.

\pgap 

Since $f_1(x)+\barf(x)=f(x)$, we have 
\[
\Var \barf(x)= 1 - \Var f_1(x).
\]

Let $X$ be a centered Gaussian r.v. with variance $(1-\sigma^2)$. Then, for $\sigma$ close to $1$, 
\begin{equation}
\begin{split}
\bE \left [X^2 \exp \left ( - \dfrac{1}{2 \sigma^2} (u-X)^2 \right ) \right ] & \leq c \sigma \exp(-1/2u^2) ( \sigma^2(1-\sigma^2)+ (1-\sigma^2)^2 u^2) \\
& \leq c (1-\sigma^2)(\sigma^2+ (1-\sigma^2)u^2) \exp(-1/2u^2).
\end{split}
\end{equation}

Replacing $X$ by $f_1(x)$ and noting that \[ \Var(f_1(x)) \leq \|r(x)\|_4^2 \] we have by expression \eqref{eq: f_1-int}
\[
\bE \left [w(f_1(x))\exp \left (- \dfrac{1}{2 \sigma^2(x)}(u-f_1(x))^2 \right) \right ] \leq c \|r(x)\|_4^2(1+ \|r(x)\|_4^2u^2) \exp(-1/2u^2).
\]

At last, we integrate out the variables $(\nabla f_2(x), \hess f_2(x))$. Note that this random vector is not Gaussian, and the second moment is of order $u^2\|r(x)\|_4^2$. By definition, each component of the vector $(\nabla f_2(x), \hess f_2(x))$ is a linear transform of $(\xi, Z)$, where the pdf is defined in \eqref{eq:def-q_u}. Now, $\bE[w(Z)| \xi=a]= \bE[\det(Z_1)w(Z_1)]$ where $Z_1$ is a Gaussian matrix distributed as 
\[
\hess f(0) | f(0)=a.
\]
Hence, if $w$ is a polynomial of degree $m$, then 
\[
\bE[w(Z,\xi)] \leq c u^m 
\]
for $u$ large enough.
Now, as pointed out earlier, the expression \eqref{eqn: cond-hess-1} is a quadratic form in $(v_2-v_1)$ with coefficients being degree one polynomials in $(u-a_1,(t-1)v_1-tv_2, M_1, M_2)$, now we have integrated out $(a_1,v_1,M_1)$.
So, after integrating other terms, the expression \eqref{eqn: cond-hess-1} is a polynomial in $v_2,M_2$ of degree at most three. 

\pgap 
Combining all the bounds above, we have that the integrand of the RHS of \eqref{eq: P_1-kac-rice} is bounded above by, 
\begin{equation}
    u^3(\|r(x)\|_4^2) \exp(-1/2u^2). 
\end{equation}

Hence, 
\[
\bE[P_1] \leq u^3 \exp(-1/2u^2) \int_{D'_R} \|r(x)\|^2_4 \td x. 
\]

\pgap

\pgap 
\textbf{Bound on $P_2$:} Define 
\[
P_2:=\{(t,x) \in [0,1]\times D'_R: \nabla F_t(x)=0, \tf(x)=u\}
\]

and we want to give a similar upper bound for $\bE[P_2]$. The computation is also similar to the upper bound for $\bE[P_1]$, so we note the changes needed to the above computation. 

We have, 
\begin{equation} \label{eq: P_2-kac-rice}
    \bE[P_2]= \int_0^1 \int_{D'_R} \bE[|\det (\td (\nabla F_t(x),\tf(x)))|\big | \nabla F_t(x)=0, \tf(x)=u] p_{t,x}(0,u) \td x \td t
\end{equation}
where $p_{t,x}$ is the pdf of $(\nabla F_t(x), \tf(x))$. Since $(f_1, f_2, \barf)$ are independent, we first condition on $f_1, f_2$ and compute the integrand above. Again, let 

\[
(f_1(x),\nabla f_1(x), \hess f_1(x))=(a_1,v_1, M_1) \quad (f_2, \nabla f_2(x), \hess f_2(x))=(a_2,v_2,M_2).
\]
Then the integrand above is, 
\begin{equation} \label{eqn: cond-hess-2}
    \begin{split}
      Z(t, v_1, M_1,a_2, v_2, M_2)=  \bE \left [|1-t||\det \begin{pmatrix}
            (v_2-v_1) & 0 \\
            \hess \barf (x) + (1-t)M_1+tM_2 & (v_2-v_1)^T 
        \end{pmatrix}| \big | \right. \\  \left . \nabla \barf(x)=(t-1)v_1-tv_2, \barf(x)=u-a_2 \right ]
    \end{split}
\end{equation}
integrated over the vectors $(f_1(x),\nabla f_1(x), \hess f_1(x))$ and $(f_2, \nabla f_2(x), \hess f_2(x))$.
First integrating out the Gaussian vector $(\barf(x), \nabla \barf(x), \hess \barf(x))$, we have that the above expression is a quadratic form in $(v_2-v_1)$ with coefficients being degree one polynomials in $(u-a_2,(t-1)v_1-tv_2, M_1, M_2)$. 

Our next goal is to compute the following quantity, 
\begin{equation} \label{eq: gauss-regression-kac-rice-2}
\bE[Z(t, \nabla f_1(x), \hess f_1(x), f_2(x), \nabla f_2(x), \hess f_2(x)) p((t-1) \nabla f_1(x)-t \nabla f_2(x), u - f_2(x))]
\end{equation}
where $p$ is now the pdf of the Gaussian vector $(\nabla \barf (x), \barf(x))$.
Observe that, since $\nabla \barf(x)$ is a zero mean Gaussian vector  
\[
p((t-1) \nabla f_1(x)-t \nabla f_2(x), u - f_2(x)) \leq p(0, u - f_2(x)). 
\]
We know that conditioning on a Gaussian random variable only decreases its variance. Hence, 
\[
p(0, u - f_2(x)) \leq \overline{P}(u - f_2(x))
\]

for $u$ large enough, where $\overline{P}$ is the pdf of $\barf(x)$. Since $\Var(\barf(x))$ is bounded away from zero uniformly for large $R$ when $x \in D'_R$, we have that $\overline{P}$ is uniformly bounded. Hence, it is enough to bound 
\[
\bE[Z(t, \nabla f_1(x), \hess f_1(x), f_2(x), \nabla f_2(x), \hess f_2(x))]. 
\]

\pgap 

 Since the expression \eqref{eqn: cond-hess-2} a quadratic form in $(v_2-v_1)$ with coefficients being degree one polynomials in $(u-a_2,(t-1)v_1-tv_2, M_1, M_2)$, we can see that

\begin{equation}
    \begin{split}
        \bE[Z(t, \nabla f_1(x), \hess f_1(x), f_2(x), \nabla f_2(x), \hess f_2(x))] & \leq u^3 \|r(x)\|_4^2
    \end{split}
\end{equation}

just like in the previous case. A noticeable difference here is that the upper bound 
\[
\bE[P_2] \leq u^3 \int_{D'_R} \|r(x)\|_4^2 \td x
\]

is missing the factor $\exp(-1/2u^2)$. 

\pgap

\pgap 

\textbf{Bound on $P_3$:} Define 
\[
P_3= \{x \in  D'_R: \nabla f(x)=0, x \in (\{f>u\}\cap \{\tf<u\}) \}. 
\]

we have, 
\begin{equation}
    \begin{split}
       \int_{D'_R} \bE \left [ \det(\hess f(x)) \bI[\{f>u\}\cap \{\tf<u\}] \bigg | \nabla f(x)=0 \right ] p_x(0) \td x 
    \end{split}
\end{equation}

where $p_x$ is the pdf of $\nabla f(x)$. 
 Again, let 

\[
(f_1(x),\nabla f_1(x), \hess f_1(x))=(a_1,v_1, M_1) \quad (f_2, \nabla f_2(x), \hess f_2(x))=(a_2,v_2,M_2).
\]

Let, 
\[
Z(a_1, v_1, M_1, a_2) = \bE[\det(\hess(\barf(x))+M_1) \bI[u-a_1< \barf(x)< u-a_2]| \nabla \barf(x)=v_1]. 
\]

Now, 
\[
\bE[[\det(\hess(\barf(x))+M_1) | \barf(x)=\bara, \nabla \barf (x)=\barv]
\]
is a degree one polynomial in $\bara$. Also, we have 
\[
\bP(u-a_1 < \barf(x) < u-a_2) \leq c |a_1-a_2|
\]

for large enough $u$, since $\barf(x)$ has pdf and it is uniformly bounded above in $x \in D'_R$.
Hence, 
\[
z(a_1,v_1,M_1,a_2)= w(u, a_1,v_1,M_1)|a_1-a_2|
\]
where $w$ is a multivariate polynomial (of degree one in $u$). 
So, finally integrating remaining variables such as $f_1(x), f_2(x)$ we have

\[
\bE[P_3] \leq c u^2 \int_{D'_R} \|r(x)\|_4 \td x. 
\]

The expectation of the following quantities also has similar upper bounds, 

\begin{equation*}
    \begin{split}
        \{x \in D'_R : \nabla f(x) =0, f(x)>u, \tf(x)<u \} \\
         \{x \in D'_R : \nabla \tf=0, f(x)>u, \tf(x)<u  \} \\
          \{x \in D'_R : \nabla \tf(x)=0, f(x) <u, \tf(x) > u \}.
    \end{split}
\end{equation*}

\pgap

\pgap

\textbf{Bound on $P_4$:} Let, 
\[
P_4= \{(t,x) \in [0,1] \times \partial(D'_R): \nabla F_t(x)=0, f(x) \geq u \}. 
\]
Note that $\partial D'_R$ consists of two connected components, namely, a circle of radius $R/\sqrt{N}* \tau(u)$ and a square of side length $R$ both centered at the origin.

First, let us compute the expected number of points hitting the circle in the flow. There's a Kac-Rice formula for the manifold version, see \cite[Chapter 12]{adler_random_2009}. To use this version, we need to calculate gradient of $\nabla F_t(x)$ when $x$ belongs to the circle ( a manifold). In our case, this corresponds to taking the derivative of $\nabla F_t(x)$ with respect to the angle $\theta$ in polar coordinates. By the Kac-Rice formula, and by the fact that Bargmann-Fock field is isotropic, we need to bound 
\[
|R \tau(u)/\sqrt{N}|\bE \left [|\det \begin{pmatrix}
    \partial_1 h(x) & \partial_2 h(x) \\
     \cdots & \cdots 
\end{pmatrix} | \bI[f(x) > u]| \nabla F_t(x)=0 \right ]
\]
where the second row of the matrix is a linear combination of the second derivatives of $f, \tf$. Ignoring the factor $\bI[f(x) > u]$ and conditioning first on $f_1, f_2$ and its derivatives, we integrate out $\barf(x)$ and its derivatives as previously done. Since 
\[
\bE[|\partial_j h(x)|] \leq c u \|r(x)\|_4 
\]
we have, 
\begin{equation}
    \begin{split}
        \int_{\partial B(0, R \tau(u)/\sqrt{N})} |R \tau(u)/\sqrt{N}|\bE \left [|\det \begin{pmatrix}
    \partial_1 h(x) & \partial_2 h(x) \\
     \cdots & \cdots 
\end{pmatrix} | \bI[f(x) > u]| \nabla F_t(x)=0 \right ] \td x  \\  \leq c R^2 \tau(u)^2N^{-1} u \|r(R \tau(u)/\sqrt{N})\|_4.
    \end{split}
\end{equation}

We now estimate an upper bound for the expected number of points hitting the square in the flow. As pointed out, it is a square of side length $R$ centered at the origin. Hence, we need to bound 
\[
\int_0^1 \int_{-R/2}^{R/2} \bE \left [\left | \det \begin{pmatrix}
    \partial_1 h(x_1,R/2) & \partial_2 h(x_1,R/2) \\
     \cdots & \cdots 
\end{pmatrix} \right| \big | \nabla F_t(x_1,R/2) =0 \right ] \td x_1 \td t. 
\]

Again using 
\[
\bE[|\partial_j h(x)|] \leq c u \|r(x)\|_4 
\]

we have
\[
\bE[P_4] \leq c R^2 \tau(u)^2N^{-1} u \|r(R \tau(u)/\sqrt{N})\|_4 + c R u \|r(R/2)\|_4
\]

The expectation of the following quantity also has the same upper bound, 
\[
\{(t,x) \in [0,1] \times \partial(D'_R): \nabla F_t(x)=0, \tf(x) \geq u \}.
\]

\pgap

\pgap 

\textbf{Bound on $P_5$:} Let 
\[
P_5= \{(t,x) \in [0,1] \times D'_R : \nabla F_t(x)=0, \det (\hess F_t(x))=0, f(x)> u, \tf(x)>u \}. 
\]

We know that, 
\[
\hess F_t(x)= \hess f(x)+t \hess h(x). 
\]
Writing $ij$ as subscripts for $ij$-th second partial derivatives, we have 

\[
\dfrac{d}{dt} \det(\hess F_t(x))= h_{11}(x) (f_{22}+th_{22})+ h_{22}(x)(f_{11}+th_{11})-2h_{12}(x)(f_{12}+th_{12}). 
\]

By the Kac-Rice formula, 
\begin{equation}
\begin{split}
\bE[P_5]= \int_0^1 \int_{D'_R} \bE \left [  | \det \begin{pmatrix}
    \partial_1 h(x) & \partial_2 h(x) & \dfrac{d}{dt} \det \hess F_t(x) \\ 
     \cdots & \cdots & \cdots 
\end{pmatrix}  | \bI[f(x)> u, \tf(x)> u] \right . \\ \left .  \bigg |  \nabla F_t(x)=0, \det (\hess F_t(x))=0 \right ] p_x(0,0) \td x \td t
\end{split}
\end{equation}

where $p_x$ is the pdf of $(F_t(x), \det(\hess F_t(x)))$. Observe that the determinant function is linear in the first row. Ignoring the factor $\bI[f(x)>u, \tf(x)> u]$, the integrand of the RHS in the above equation is linear in $h(x)$ and its derivatives (up to second derivatives). Hence, integrating out $\barf(x)$ and its derivatives as previously done, we have
\[
\bE[P_5] \leq c u \int_{D'_R} \|r(x)\|_4 \td x.
\]

\subsubsection{Gathering all the bounds}

From the expectation bounds on the number of points $P_1,\ldots, P_5$, we have
\begin{equation}
    \begin{split}
        \bE||\Psi_{R,u}^0|-|\Psi_{R,u}|| \leq c \left ( u^3 \exp \left( -\dfrac{1}{2}u^2 \right) \int_{D_{R,u}'} \|r(x)\|_4^2 \td x \right . \\
       u^3 \int_{D_{R,u}'}  \|r(x)\|_4^2 \td x + u^2 \int_{D_{R,u}'} \|r(x)\|_4 \td x + \\
      \left . u \int_{D_{R,u}'} \|r(x)\|_4 \td x + u^2 \|r(u)\|_4+ Ru \|r(R/2)\|_4 \right ).
    \end{split}
\end{equation}

Recall that 
\begin{equation}
    \begin{split}
        \|r(x)\|_4 &= \max_{|\alpha| \leq 4 } \|\partial^{\alpha} r(x)\| \\
        & \leq c \|x\|^4 \exp \left ( -\dfrac{1}{2} \|x\|^2\right). 
    \end{split}
\end{equation}

By a polar coordinate transformation in 2-dim, we have
\begin{equation}
    \begin{split}
        \int_{D_{R,u}'} \|x\|^4 \exp(-1/2\|x\|^2) \td x & \leq c \int_u^{\infty} t^{5} \exp(-1/2 t^2) \td t \\
        & \leq c' u^4 \exp(-1/2u^2)
    \end{split}
\end{equation}
for large enough $u$.

Now, by assumption, $u \leq \log R$, hence
\[
R \exp(-R^2/4) \leq c' \exp(-u^2/2). 
\]

Hence, we have
\begin{equation} \label{eqn:final-coupling-bound}
    \bE||\Psi_{R,u}^0|-|\Psi_{R,u}|| \leq c u^6 \exp(-1/2u^2).
\end{equation}

Also, 
\begin{equation} \label{eqn:volume-bound}
    \vol(A_0)= \tau(u)^2 \leq c u^3 \exp(-1/2u^2). 
\end{equation}

Finally substituting the bounds \eqref{eqn-bound-cluster-points},\eqref{eqn:final-coupling-bound}, \eqref{eqn:volume-bound} to the RHS of \eqref{eqn:chen-stein-bound} we have
\[
d_{TV}(|\CL \Psi_{R,u}|, U_{R,u}) \leq C \left( u \exp \left ( - u^2\dfrac{\delta}{2(2-\delta)}\right) +  u^3 \exp(-1/2u^2)+ u^6 \exp(-1/2u^2)\right). 
\]

Hence for $\beta < \min \{ \dfrac{\delta}{2(2-\delta)}, 1/2\}$ we have, 
\[d_{TV}(|\CL \Psi_{R,u}|,U_{R,u} \leq C \exp(-\beta u^2).\]

\pgap 

This finishes the proof of Theorem \ref{thm-number-count}.

\subsection{Proof of Theorem \ref{thm-supercritical-level}}

\begin{proof} Recall that $I_n=[0,n]^d \cap \bZ^d, D=[-1/2,1/2]^d, D_t= t+D$ for $t \in I_n$. For $u(n)=u>0$, let  
\[
X_{t,n}=X_t:=\bI \left [\max_{x \in D_t} f(x) > u \right ] .
\]
Since $X' \equiv 0$ be the zero process, we have
\[
\|\CL(X)-\CL(X')\|_{TV} = \bP \left ( \bigcup_{t \in I_n} \{X_t=1  \} \right).
\]
Denote $p_t=\bP(X_t=1)$, $p_{st}=\bP(X_s=1,X_t=1)$. By inclusion-exclusion principle (also called the Bonferroni inequality), 
\begin{equation} \label{eq:bonferroni}
    \sum_{I_n} p_t - \dfrac{1}{2}\sum_{s \neq t} p_{st} \leq \|\CL(X)-\CL(X')\|_{TV} \leq \sum_{I_n} p_t
\end{equation}

By Theorem \ref{thm-piterbarg-excursion}, we have 
\[
p_t = c_0 \cdot u^{d-1} \exp(-u^2/2)(1+o(1))
\]
for all $t \in I_n$ using the stationarity of the field $f$. By plugging $u= \sqrt{2\alpha d \log n}$, we get 
\[
p_t= c_0 n^{-\alpha d} (\log n)^{(d-1)/2}(1+o(1)).
\]

Notice that 
\[\mu_n:= \bE[ \text{ number of } t \in I_n \text{ with } X_t=1]= \sum_{I_n} p_t= n^d p_0.\]

Now we show that $ \sum_{s \neq t}p_{st} =o(\mu_n)$. Let $\zeta>0$ be an arbitrary number. We give an upper bound for $p_{st}$ by considering the cases $|s-t| \leq \sqrt{d}$, $ \sqrt{d}< |s-t| \leq n^{\zeta}$ and $|s-t| \geq n^{\zeta}$.  We have 

\begin{equation} \label{eq:double-sum}
    \begin{split}
        p_{st}=& \bP(X_s=1,X_t=1) \\
        =& \bP \left ( \max_{D_s} f > u, \max_{D_t} f >u \right ) \\
       \leq & \bP \left ( \max_{x \in D_s, y \in D_t} f(x)+f(y) > 2u \right ).
    \end{split}
\end{equation}

Now we use Theorem \ref{thm-piterbarg-excursion} for the field $F(x,y)=f(x)+f(y)$.
Note that if $h={\rm{dist}}(D_s,D_t)>0$ then 
\[
{\rm{Var}}(f(x)+f(y)) \leq 4 -2 \beta ,
\]
for all $x \in D_s, y \in D_t$, where 
\[
\beta= \min_{|x-y|\geq h} (1-r(x,y)) >0.
\]

Hence, we have, 
\[
p_{st} \leq c u^{d-1} \exp \left (-\dfrac{4u^2}{2(4-2\beta)}\right ) = c u^{d-1} \exp \left (-\dfrac{u^2}{2}\times \dfrac{1}{1-\beta /2}\right ).
\]

For the case $\sqrt{d} < |s-t| \leq n^{\zeta}$, there exists a constant $ \delta >0$ depending only on the covariance $r$ such that $\beta > \delta$ for all $n$.

Hence, 
\begin{align}\label{eq:middle-bound}
     \sum_{\sqrt{d} < |s-t| \leq n^{\zeta}} p_{st} & \leq c n^d n^{\zeta d} u^{d-1} \exp \left (-\dfrac{u^2}{2}\times \dfrac{1}{1- \delta/2}\right ) \\
     & \leq c n^d n^{ \zeta d} n^{- \frac{\alpha d}{1-\delta/2}}(\log n)^{(d-1)/2} \\
     & \leq c n^{(1- \alpha - \delta_0)d} = o(\mu_n)
\end{align}

for $\zeta < \delta/2$ and a fixed $0< \delta_0 < \delta/2$. 

Let 
\[ \Tilde{r}(h):= \max_{|x-y|\geq h} r(x-y) \]
and notice that $\Tilde{r}(h) \to 0$ as $h \to \infty$ by our assumption. We have 

\begin{align}
    \sum_{|s-t| \geq n^{\zeta}} p_{st} & \leq n^d(n^d-n^{\zeta d}) n^{- \frac{2 \alpha d}{1+\Tilde{r}(n^\zeta)}} \\
    & \leq n^{2(1-\alpha)d+ o(1)} = o(\mu_n).
\end{align}

Finally we bound the sum of $p_{st}$ for the case $|s-t| \leq \sqrt{d}$. This is the case where $dist(D_s,D_t)=0$, i.e. the squares $D_s$ and $D_t$ are touching. Fix $\ep >0$ and let $D^\ep_{s,t}$ denote the union of $\ep$-neighborhoods of boundaries of $D_s$ and $D_t$. Let $Q^\ep_t$ denote $D_t$ with $\ep$-neighborhood of boundary of $D_t$ removed. The event $X_s=1,X_t=1$ can happen if for some $x \in D^\ep_{s,t}, f(x)> u$ or that $f$ exceeds $u$ in both $Q^\ep_s$ and $Q^\ep_t$.
Hence, by the union bound on those two events, 
\begin{align}
    p_{st} \leq c_1 \ep n^{-\alpha d} (\log n)^{(d-1)/2} + c_2 n^{- \frac{\alpha d}{1-\delta_1}}(\log n)^{(d-1)/2}
\end{align}
for some $\delta_1 >0$ which depends only on $\ep$. Since $1/(1-\delta_1)> 1+\delta_1$ for small $\delta_1$, for large enough $n$, 
\begin{equation}
    \sum_{0<|s-t| \leq \sqrt{d}} p_{st} \leq c_3 \ep  n^{(1-\alpha)d - \delta_2} (\log n)^{(d-1)/2} = o(\mu_n)
\end{equation}
where $\delta_2 = \alpha d \delta_1$. By \eqref{eq:bonferroni}, we have 
\[
\mu_n - |o(\mu_n)| \leq \|\CL(X)-\CL(X')\|_{TV} \leq \mu_n.
\]
\end{proof}

\appendix

\section{Method of comparison}

One of the fundamental questions which pops up regularly when studying the excursion of Gaussian fields is the following: given two Gaussian vectors in $\bR^d$ with a covariance matrix that is close enough, how close are the excursion probabilities? 
Here we state a generalisation of the classic Berman's inequality, taken from \cite{piterbarg_asymptotic_1996}.

\pgap 

Say we're given $n$ sequence of real numbers, that we call \textit{discritising levels}, 
\[
\mathbf{u}(k)=(\cdots < u_{-1}(k)<u_0(k)<u_1(k)< \cdots) \quad k=1,2,3,\ldots n.
\]
Consider the $\sigma-$algebra $\mathcal{U}$ generated by the $n$-dim rectangles 
\[
\Pi_{\mathbf{i}}=\{ (x_1,x_2,\ldots, x_n): x_k \in [u_{i_k}, u_{i_k+1}(k) \}
\]
where $\mathbf{i}=(i_1,\ldots,i_n) \in \bZ^n$ is a multi-index.
Let $\mathbf{X}_0=(X_0(1),X_0(2),\ldots,X_0(n))$ and $\mathbf{X}_1=(X_1(1),X_1(2),\ldots,X_1(n))$ be two independent Gaussian vectors in $\bR^n$ with zero mean. Consider an interpolation of these vectors, 
\[
\mathbf{X}_h=\sqrt{h}\mathbf{X}_1+\sqrt{1-h}\mathbf{X}_0 \quad 0\leq h\leq 1.
\]
Denote by $R_h=\{r_h(i,j): 1\leq i,j\leq n\}$ the covariance matrix of $\mathbf{X}_h$.

\begin{theorem}[Theorem 1.2, \cite{piterbarg_asymptotic_1996}] \label{thm-method-of-comparision}
    With notations as above, if $r_0(k,k)=r_1(k,k)$ for all $k$ and $|r_0(k,l)|<1$ for $k\neq l$, then for any $B \in \mathcal{U}$, we have
    \[
    |\bP(\mathbf{X}_0 \in B)-\bP(\mathbf{X}_1 \in B)|\leq 2 \sum_{k>l}^n|r_0(k,l)-r_1(k,l)| \sum_{i,j} \int_0^1 \phi(u_i(k),u_j(k);r_h(k,l)) dh
    \]
where $\phi(x,y;r)$ is the density of 2-dim Gaussian with covariance $r$.
\end{theorem}

\section{Maximum of Gaussian fields}

It is a classical fact in probability that the expected maximum of $n$ i.i.d. standard Gaussian random variables behaves asymptotically like $\sqrt{2\log n}$ as $n \to \infty$. Also, it can be shown that, even if the random variables are dependent, it cannot exceed $\sqrt{2 \log n}$.
What is a bit surprising is that a large number of Gaussian fields models with correlation decay `fast enough' also have exact asymptotic $\sqrt{2 \log n}$. Examples include 2-dim discrete Gaussian free field, Sherrington-Kirkpatrick model energy landscape, etc \cite{chatterjee_superconcentration_2016}. 
They are expected to converge to the Gumbel distribution for the asymptotic distribution of the (centered, normalised) maximum.

\pgap 

We have the same asymptotic for stationary smooth Gaussian fields. 

\begin{theorem}[Theorem 14.1, \cite{piterbarg_asymptotic_1996}]
    Let $X:\bR^d \to \bR$ be centered, unit variance, $C^2-$smooth Gaussian field with covariance $r(t)=\bE[X(0)X(t)]$.
    Assume that, for some $\alpha>0$, 
    \[
    \int_{\bR^d} |r(t)|^{\alpha} dt < \infty. 
    \]
    Then, 
    \[
    \bP \left ( \max_{t \in [0,R]^d} (X(t)-l_R)l_R<x\right ) =\exp(-\exp(-x))
    \]
    where $l_R$ is the largest solution of the equation 
    \[
    \dfrac{R^d\det (\Lambda_X)^{1/2}}{(2 \pi)^{d-1}}l^{d-1}\exp(-l^2/2)=1
    \]
    and $\Lambda_X$ is the covariance matrix of $\nabla X(0)$.
\end{theorem}

From this theorem, we can get the exact asymptotic of the expected mean, 
\[
\dfrac{\bE[\max_{ t \in [0,R]^d} X(t)]}{\sqrt{2d \log R}} \to 1 \qquad {\rm{as}} \quad R \to \infty.
\]

\section{Discrete approximation of excursion probability} \label{sec: discrete-approx-appendix}

Here we show that for a $C^1$-smooth stationary Gaussian field $f : \bR^d \to \bR$ the excursion probability of $f>u$ on a box $B$ approximated by that on the grid $B \cap u^{-1}\bZ^d$ with high probability. These approximations were first proved by Pickands \cite{pickands_upcrossing_1969} and later generalised to what is now called the double sum method (see Section 6 of \cite{piterbarg_asymptotic_1996} for a discussion). 

\pgap 

Let $\chi : \bR^d \to \bR$ be a continuous Gaussian field with mean $\bE[\chi(s)]=-\|s\|^2$ and covariance
\[\cov(\chi(t), \chi(s))= 2 \langle t,s \rangle \]

where $\langle \cdot , \cdot \rangle $ is the usual inner product on $\bR^d$. The existence of the continuous trajectories (in fact, smooth trajectories) of $\chi$ follows from \cite[Section A.9]{nazarov_asymptotic_2016}.  We now use the following lemma from \cite{piterbarg_asymptotic_1996} for the special case of $C^1$-smooth Gaussian fields. Lemma 6.1 of \cite{piterbarg_asymptotic_1996} holds for much more general continuous stationary Gaussian fields  such as those with $\alpha$-Holder continuous paths with $\alpha \in (0,1)$ (e.g. Ornstein-Uhlenbeck processes on $\bR^d$)

\begin{lemma}[Lemma 6.1 of \cite{piterbarg_asymptotic_1996}]
    Let $f: \bR^d \to \bR$ be a zero-mean, stationary, $C^1$-smooth Gaussian field. Assume that the covariance satisfies 
    \[ \bE[f(x)f(0)]= 1- \|x\|^2 +o(\|x\|^2) \quad \text{ as } x \to 0.\]
    Then for any compact set $M \subset \bR^d$,  
    \[
    \bP \left( \max_{x \in u^{-1}M} \dfrac{f(x)}{1+\|x\|^2} > u  \right )= \dfrac{1}{\sqrt{2 \pi}}C(M) u^{-1} e^{-u^2/2}(1+o(1)) \quad \text{ as } u \to \infty
    \]
    where \[ C(M)= \bE \left [\exp \left (\max_{t \in M} (\chi(t)-\|t\|^2) \right ) \right ] < \infty . \]
\end{lemma}

The `local structure of the covariance' can be achieved by rescaling the domain. That is, for any $C^1$-smooth stationary centered field $f$ with $f(0)$ having unit variance, there exists a fixed invertible matrix $A$ such that 
\[\bE[f(Ax)f(0)]=1-\|x\|^2+o(\|x\|^2).\]

\pgap 

Observing that 
\begin{align*}
    C(M)= & \bE  \left [\exp \left (\max_{t \in M} (\chi(t)-\|t\|^2) \right ) \right ] \\
    &=  \int_{0}^\infty e^v \bP \left( \sup_{M} (\chi(t) - \|t\|^2) > v \right ) dv
\end{align*}

the proof of Lemma 6.1 in $\cite{piterbarg_asymptotic_1996}$ yields the following, for any fixed $b>0$, 

\begin{multline} \label{eq:discrete-approx}
\lim_{ u \to \infty} \dfrac{1}{\sqrt{2 \pi}} u e^{u^2/2} \bP \left ( \sup_{u^{-1}M \cap (b u^{-1}\bZ^d)} f(x) \leq u, \sup_{u^{-1}M} f(x) > u \right) = \\ \int_0^{\infty} e^v \bP \left ( \sup_{M \cap (b \bZ^d)} \chi(t) - \|t\|^2 \leq  v, \sup_{M} \chi(t) - \|t\|^2 > v \right ) dv.
\end{multline}

\section{Palm measures of point processes} \label{appendix-palm}

Palm measures provide a rigorous framework for analysing point processes from the perspective of a typical point. While the distribution of a point process $\Phi$ on a measurable space $S$ describes the statistical behaviour of configurations of points, the \emph{Palm measure} conditions on the occurrence of a point at a specific location and allows us to understand the structure surrounding such points. References for this section of the appendix are Kallenberg's classical book \cite{kallenberg_random_2017} and Last and Penrose's book \cite{last_lectures_2017}. 

\pgap 

Let $(\Omega, \mathcal{F}, \mathbb{P})$ be a probability space, and let $\Phi$ be a simple point process on a Borel space $S$ (typically $S = \mathbb{R}^d$), taking values in the space $\mathcal{N}(S)$ of locally finite counting measures on $S$. For a measurable function $f: S \times \mathcal{N}(S) \to [0, \infty)$, the \emph{Campbell formula} states:
\[
\mathbb{E} \left[ \sum_{x \in \Phi} f(x, \Phi) \right] = \int_S \mathbb{E}[f(x, \Phi)] \, \lambda(dx),
\]
where $\lambda$ is the intensity measure of $\Phi$.

\pgap 

If $\Phi$ is a \emph{stationary} point process on $\mathbb{R}^d$ with constant intensity $\rho > 0$, then the \emph{Palm measure} $\mathbb{P}^x$ is defined as the probability law of the process conditioned to have a point at $x$. The Palm-Campbell formula becomes
\[
\mathbb{E} \left[ \sum_{x \in \Phi} f(x, \Phi) \right] = \int_{\mathbb{R}^d} \mathbb{E}^x[f(x, \Phi)] \, \rho \, dx,
\]
where $\mathbb{E}^x$ denotes expectation under $\mathbb{P}^x$.

Under stationarity, studying the Palm measure at the origin, denoted $\mathbb{P}^0$, suffices. For measurable functionals $f : \mathcal{N}(\mathbb{R}^d) \to \mathbb{R}$, the Palm measure can be formally defined via disintegration:
\[
\mathbb{E} \left[ \sum_{x \in \Phi} f(\theta_x \Phi) \right] = \rho \int_{\mathbb{R}^d} \mathbb{E}^{x}[f(\Phi)] \, dx,
\]
where $\theta_x \Phi$ is the translation of $\Phi$ by $-x$.

\subsection{Poisson Process Characterisation via Stein's Method}

The Poisson point process plays a central role in the theory of stochastic geometry and spatial statistics, serving as the canonical model for complete spatial randomness. Various techniques exist for approximating more complex point processes using a Poisson process. One of the most powerful and flexible is \emph{Stein's method}, introduced initially for approximating the Poisson distribution, and later extended to spatial point processes. This section provides a mathematical overview of the Poisson process characterisation via Stein’s method, and its connections to Palm measures.

\pgap

Let $W$ be an integer-valued random variable with $\mathbb{E}[W] = \lambda$. The goal of Poisson approximation is to compare the distribution of $W$ to that of a Poisson random variable $Z \sim \mathrm{Poisson}(\lambda)$. The classical {Stein–Chen method} characterises the Poisson distribution as the unique solution to the identity
\[
\mathbb{E}[\lambda f(W + 1) - W f(W)] = 0 \quad \text{for all suitable test functions } f : \mathbb{N} \to \mathbb{R}.
\]
This identity provides a means of bounding the total variation distance between the distribution of $W$ and $\mathrm{Poisson}(\lambda)$:
\[
d_{\mathrm{TV}}(\mathcal{L}(W), \mathrm{Poisson}(\lambda)) \leq \sup_{f \in \mathcal{F}} |\mathbb{E}[\lambda f(W + 1) - W f(W)]|,
\]
where $\mathcal{F}$ is a class of bounded test functions. Chen \cite{chen_poisson-1975} developed a general method for applying this identity, especially when $W = \sum_{i=1}^n X_i$ is a sum of (possibly dependent) Bernoulli random variables.

\pgap 

Barbour and Brown \cite{barbour_poisson_1992} and later Schuhmacher and coauthors extended Stein's method to the setting of point processes. Let $\Phi$ be a point process on a Polish space $S$ with finite intensity measure $\lambda(dx)$. We wish to compare the law of $\Phi$ to that of a Poisson point process $\Pi$ with the same intensity measure.

\pgap 

A fundamental observation is that the Poisson process is characterised as the unique point process $\Phi$ satisfying:
\[
\mathbb{E} \left[ \int_S f(x, \Phi + \delta_x) \, \lambda(dx) \right] = \mathbb{E} \left[ \sum_{x \in \Phi} f(x, \Phi) \right],
\]
for all suitable test functions $f: S \times \mathcal{N}(S) \to \mathbb{R}$. This identity is equivalent to the statement that the \emph{reduced Palm measure} $\mathbb{P}^{!x}$ coincides with the original measure $\mathbb{P}$ for all $x$, a defining property of the Poisson process:
\[
\mathbb{P}^{!x} = \mathbb{P}, \quad \text{for all } x \in S.
\]

\bibliographystyle{alpha}
\bibliography{references}

\end{document}